\documentclass[11pt]{article}
\usepackage{amsmath,amsthm,amssymb}
\usepackage{mathtools}
 \usepackage{cite}
 \usepackage{color}
 \usepackage{diagbox}
\usepackage[twoside]{geometry}
\usepackage{graphicx}
\usepackage{epsfig}   
\usepackage[T1]{fontenc}
 \usepackage{extarrows}
 \usepackage{textcomp}
 \usepackage{stmaryrd}
 \usepackage[toc,page]{appendix}
\usepackage[numbers,sort&compress]{natbib}
\usepackage{amssymb,amscd,mathrsfs,wasysym} 

\usepackage[T1]{fontenc}

\newtheorem{theorem}{Theorem}[section]
\newtheorem{proposition}[theorem]{Proposition}
\newtheorem{lemma}[theorem]{Lemma}
\newtheorem{corollary}[theorem]{Corollary}
\newtheorem{definition}[theorem]{Definition}

\newtheorem{ass}[theorem]{Assumption}
\numberwithin{equation}{section} 
\numberwithin{figure}{section}  

\newcommand\N{\mathbb{N}}
\newcommand\R{\mathbb{R}}
\newcommand\A{\mathbb{A}}
\newcommand\C{\mathbb{C}}
\newcommand\Z{\mathbb{Z}}

\newcommand\cA{\mathcal{A}}
\newcommand\cS{\mathcal{S}}
\newcommand\cM{\mathcal{M}}
\newcommand\cR{\mathcal{R}}

\newcommand\dps{\displaystyle}
\newcommand\eps{\varepsilon}

\newcommand{\id}{{\normalfont\hbox{1\kern-0.15em \vrule width .8pt depth-.5pt}}}

\let\oldmarginpar\marginpar
\renewcommand\marginpar[1]{\-\oldmarginpar[\raggedleft\footsize #1]%
{\raggedright\footsize #1}}
\begin{document}

\title{Embedded corrector problems for homogenization in linear elasticity} 
\author{Virginie Ehrlacher$^{1,3}$, Fr\'ed\'eric Legoll$^{2,3}$, Benjamin Stamm$^4$, Shuyang Xiang$^{1,4}$
\\
{\footnotesize $^1$ CERMICS, \'Ecole des Ponts, Marne-La-Vall\'ee, France}
\\
{\footnotesize $^2$ Navier, \'Ecole des Ponts, Univ Gustave Eiffel, CNRS, Marne-La-Vall\'ee, France}
\\
{\footnotesize $^3$ MATHERIALS project-team, Inria, Paris, France}
\\
{\footnotesize $^4$ MATHCCES, Department of Mathematics, RWTH Aachen University, Schinkelstrasse 2, D-52062 Aachen, Germany}
}

\maketitle

\begin{abstract}
%
In this article, we extend the study of embedded corrector problems, that we have previously introduced in the context of the homogenization of scalar diffusive equations, to the context of homogenized elastic properties of materials. This extension is not trivial and requires mathematical arguments specific to the elasticity case. Starting from a linear elasticity model with highly-oscillatory coefficients, we introduce several effective approximations of the homogenized tensor. These approximations are based on the solution to an embedded corrector problem, where a finite-size domain made of the linear elastic heterogeneous material is embedded in a linear elastic homogeneous infinite medium, the constant elasticity tensor of which has to be appropriately determined. The approximations we provide are proven to converge to the homogenized elasticity tensor when the size of the embedded domain tends to infinity. Some particular attention is devoted to the case of isotropic materials. 
\end{abstract}


\section{Introduction}

Let $d \in \N^\star$ and $D \subset \R^d$ be a bounded, regular domain. For a given function $f \in L^2(D)^d$, we consider the following highly-oscillatory linear elasticity problem: find $u^\eps \in H^1_0(D)^d$ such that
$$
- {\rm div} \big( \A^\eps e(u^\eps) \big) = f \qquad \mbox{in $D$}, 
$$
where the strain tensor $e(u_\eps)$ is 
$$
e(u^\eps) = \frac{1}{2} \Big( \nabla u^\eps + \big( \nabla u^\eps \big)^T \Big), 
$$
and where the elasticity tensor-valued field $\A^\eps = \left( \A^\eps_{ijkl} \right)_{1 \leq i,j,k,l \leq d} \in L^\infty(D,\R^{d \times d \times d \times d})$ varies at the small characteristic length-scale $\eps$ (the elasticity tensor $\A^\eps$ additionally satisfies standard symmetry and coercivity assumptions that are made precise in Section~\ref{sec:pre} below). 

It follows from classical results in homogenization theory (see e.g.~\cite{CD,JKO} and~\cite[Chapter~1]{Allaire_2002_shape_opt}) that, under appropriate uniform boundedness and ellipticity assumptions on the family $\left(\A^\eps\right)_{\eps>0}$, there exists a subsequence $\left(\A^{\eps'}\right)_{\eps'>0}$ which converges in the sense of homogenization to some field $\A^\star \in L^\infty(D,\R^{d \times d \times d \times d})$. More precisely, for any $f \in L^2(D)^d$, we have that $(u^{\eps'})_{\eps'>0}$ (respectively $\left( \A^{\eps'} e(u^{\eps'}) \right)_{\eps'>0}$) weakly converges in $H_0^1(D)^d$ (respectively in $L^2(D)^{d \times d}$) to $u^\star$ (respectively to $\A^\star e(u^\star)$), where $u^\star$ is the unique solution in $H^1_0(D)^d$ to
$$
- {\rm div} \big( \A^\star e(u^\star) \big) = f \qquad \mbox{in $D$},
$$
and where $\A^\star$ is the homogenized elasticity tensor field.

In the following, we consider the case where $\A^\eps = \A(\cdot/\eps)$ for a fixed tensor $\A$, and where the {\em whole} sequence $\left( \A^\eps \right)_{\eps>0}$ (and not only a subsequence) converges in the sense of homogenization to some {\em constant} homogenized tensor, that we denote $A^\star$ in the sequel. This setting of course includes the periodic setting (when $\A$ is $\Z^d$ periodic), the quasi-periodic setting, and the ergodic random stationary setting (in which the constant homogenized tensor $A^\star$ is deterministic), to name but a few. 

In general (except in the periodic setting), the computation of $A^\star$ is challenging. It usually requires the resolution of some auxiliary problems, called the corrector problems, which are defined on the {\em whole space} $\R^d$. Standard numerical methods to approximate $A^\star$ are often based on the introduction of a large but finite simulation domain (sometimes called the representative volume element) and require the resolution of an approximate corrector problem defined on this truncated domain (and complemented by appropriate, e.g. periodic, boundary conditions). The larger the domain is, the more accurate the obtained approximation of $A^\star$ is. However, for large representative volume elements, the numerical resolution of the truncated corrector problems generally yields discretized problems with a very large number of degrees of freedom, which lead to prohibitive computational costs. 

Motivated by our previous works~\cite{notre_cras,PART1,PART2} in the context of a simple, highly oscillatory diffusive equation (which can be thought of as modelling the thermal properties of heterogeneous materials), we consider in this article a new version of the corrector problem (which we call the embedded corrector problem) in the context of linear elasticity. This problem is defined over the whole space, where a representative volume element of the highly oscillatory material is embedded in an infinite homogeneous medium, the elasticity tensor of which has to be properly selected. We introduce four possible choices for the value of the exterior tensor, which provide, along with the embedded corrector problem, four consistent approximations of the homogenized tensor $A^\star$. More precisely, we prove here that the effective tensors computed with these four methods all converge to $A^\star$ as the size of the representative element goes to infinity. The motivation for considering such embedded problems relies in the fact that a very efficient numerical method can be proposed for their resolution in the practically relevant case of materials composed of spherical inclusions. In other words, despite the fact that the embedded problem is posed on the whole space, it can be efficiently solved using the specific geometrical structure of the heterogeneous material and the fact that, beyond the representative volume element, the material properties are homogeneous. We refer to~\cite{PART2} for more details on the numerical method for the case of scalar diffusive equations. 

The present article is thus an extension of~\cite{PART1} to the elasticity case (and we refer to~\cite{SX} for the companion, numerically-oriented contribution). We stress the fact that the extensions of several results of~\cite{PART1} are actually not straightforward and involve new strategies of proof. We highlight below the results which require a specific treatment in the elasticity case, compared to the thermal case.

\medskip

The article is organized as follows. In Section~\ref{sec:pre}, we recall some well-known theoretical results in homogenization theory and some classical numerical methods to approximate $A^\star$. The embedded corrector problem we consider in this article is introduced in Section~\ref{sec:problem} (see problem~\eqref{p-corrector}). The different effective approximations of the homogenized tensor which can be computed using this embedded corrector problem are presented in Section~\ref{sec:Alternative}, together with the proof of their consistency. Our first main result is therefore Proposition~\ref{prop:convergence}. Stronger results can be obtained in the case when the heterogeneous materials (modelled by $\A^\eps$) are isotropic. To prove these results, we need to introduce some theoretical elements on linear elasticity problems involving isotropic materials, such as the definition of vectorial spherical harmonics and the resolution of the Eshelby problem. These ingredients are introduced in Section~\ref{sec:iso} (see in particular Lemma~\ref{Elsheby}). The stronger results related to embedded corrector problems for the homogenization of isotropic materials are presented and proved in Section~\ref{sec:pro}. This yields our second main result, Proposition~\ref{self-consistent}, and also Proposition~\ref{prop:convergence_iso}. In Appendix~\ref{sec:appendix}, we collect some auxiliary results related to Korn's inequalities which are used throughout the article.


\section{Theoretical preliminaries} \label{sec:pre}

Let $d \in \N^\star$ be the ambient dimension. Throughout this article, we make use of the following notation: for any $\sigma = (\sigma_{ij})_{1 \leq i,j \leq d} \in \R^{d \times d}$, $\tau = (\tau_{ij})_{1 \leq i,j \leq d} \in \R^{d \times d}$ and $A = (A_{ijkl})_{1 \leq i,j,k,l \leq d} \in \R^{d \times d \times d \times d}$, we set
\begin{align*}
  \sigma \cdot \tau & := \sum_{1 \leq i,j \leq d} \sigma_{ij} \, \tau_{ij},
  \\
  |\sigma| & := \sqrt{\sigma \cdot \sigma},
  \\
  A\sigma & := \big( (A\sigma)_{ij} \big)_{1 \leq i,j \leq d} \quad \mbox{ where } \quad (A\sigma)_{ij} = \sum_{1 \leq k,l \leq d} A_{ijkl} \, \sigma_{kl} \quad \mbox{for all $1\leq i,j \leq d$.}
\end{align*}
Let $\mathcal{S}$ denote the set of real-valued symmetric matrices of size $d \times d$ and $\widetilde{\mathcal{M}}$ the set of fourth-order tensors $A = (A_{ijkl})_{1 \leq i,j,k,l \leq d} \in \R^{d \times d \times d \times d}$ satisfying the following symmetry relationships:
$$
\forall 1 \leq i,j,k,l \leq d, \qquad A_{ijkl} = A_{klij} = A_{jikl} = A_{ijlk}.
$$
The set $\widetilde{\mathcal{M}}$ can be easily seen as the set of symmetric endomorphisms of $\mathcal{S}$. 

\medskip

Let $0 < \alpha < \beta < +\infty$ and let $\mathcal{M} \subset \widetilde{\mathcal{M}}$ be the set of fourth-order tensors $A \in \widetilde{\mathcal{M}}$ satisfying the following ellipticity and boundedness assumptions:
\begin{equation}\label{eq:ellipticity}
\forall \sigma \in \mathcal{S}, \qquad \alpha \, |\sigma|^2 \leq \sigma \cdot A \sigma \leq \beta \, |\sigma|^2. 
\end{equation}
It is easy to see that any $A \in \mathcal{M}$ satisfies $| A_{ijkl} | \leq \beta$ for any $1 \leq i,j,k,l \leq d$.  

\medskip

We also introduce a canonical basis $(\sigma^{ij})_{1 \leq i \leq j \leq d}$ of $\mathcal{S}$ as follows: for all $1 \leq i \leq j \leq d$, we set $\sigma^{ij} = (\sigma^{ij}_{kl})_{1 \leq k,l \leq d}$ with
\begin{equation}\label{Sij}
\forall 1 \leq k,l \leq d, \qquad \sigma^{ij}_{kl} := \frac{1}{2} (\delta_{ik} \delta_{jl} + \delta_{il} \delta_{jk}), 
\end{equation}
where $\delta$ stands for the Kronecker symbol. For example, in the case when $d=2$, we have $\dps \sigma^{11} = \begin{bmatrix} 1 & 0 \\ 0 & 0 \end{bmatrix}$ and $\dps \sigma^{12} = \sigma^{21} = \begin{bmatrix} 0 & 1/2 \\ 1/2 & 0 \end{bmatrix}$.

Lastly, for any regular enough function $u = (u_i)_{1 \leq i \leq d}$ with values in $\R^d$, we denote by $e(u) = \left( e_{ij}(u) \right)_{1 \leq i,j \leq d}$ its associated strain tensor defined by
$$
\forall 1 \leq i,j \leq d, \qquad e_{ij}(u) := \frac{1}{2} \left( \partial_{x_j} u_i + \partial_{x_i} u_j \right). 
$$

\subsection{Classical homogenization results}

We start by recalling the definition of G-convergence (see the seminal work~\cite{Murat,cherkaev} and also~\cite[Definition~12.2]{JKO} and~\cite[Chapter~1]{Allaire_2002_shape_opt}):

\begin{definition}\label{Def-G}
Let $D \subset \R^d$  be a bounded smooth domain and $\left( \A^\eps \right)_{\eps>0} \subset L^\infty(D,{\cal M})$. The sequence $(\A^\eps)_{\eps>0}$ is said to G-converge to some $\A^\star \in L^\infty(D,{\cal M})$ if, for any $f\in H^{-1}(D)^d$, the solutions $u^\eps \in H_0^1(D)^d$ and $u^\star \in H_0^1(D)^d$ to the problems 
\begin{gather*}
- {\rm div} \big(\A^\eps e(u^\eps) \big) = f \qquad \text{in ${\mathcal{D}}'(D)^d$}, 
\\
- {\rm div} \big(\A^\star e(u^\star) \big) = f \qquad \text{in ${\mathcal{D}}'(D)^d$},
\end{gather*}
satisfy the following two properties:
\begin{gather*}
  u^\eps \xrightharpoonup[\eps \to 0]{} u^\star \quad \text{weakly in $H_0^1(D)^d$},
  \\
  \A^\eps e(u^\eps) \xrightharpoonup[\eps\to 0]{} \A^\star e(u^\star) \quad \text{weakly in $L^2(D,\mathcal{S})$}. 
\end{gather*}
The tensor $\A^\star$ is called the homogenized limit of $(\A^\eps)_{\eps>0}$. 
\end{definition}

\medskip

A particular motivation for our work is the computation of effective elastic properties in stochastic homogenization. Let us recall some well-known results in the ergodic stationary setting. We restrict our presentation here to the case of continuous stationarity for the sake of simplicity and refer to~\cite[Chapter~7]{JKO} for a more general presentation (another type of stationarity, namely discrete stationarity, is described in e.g.~\cite{singapour}). Let $(\Omega, \mathcal{F}, \mathbb{P})$ be a probability space and let us assume that the group $(\R^d, +)$ acts on $\Omega$. Let us denote its action by $(T_x)_{x\in \R^d}$ and assume that
\begin{itemize}
\item it preserves the measure $\mathbb{P}$, i.e. for any $F \in \mathcal{F}$ and any $x \in \R^d$, we have $\mathbb{P}(T_x^{-1} F) = \mathbb{P}(F)$; 
\item it is ergodic, i.e. for any $F \in \mathcal{F}$, if $T_x(F) = F$ for all $x \in \R^d$, then $\mathbb{P}(F)$ is equal to $1$ or $0$. 
\end{itemize}
Let us recall the definition of a stationary field.

\begin{definition}
\label{De:stationary}
A random field $a \in L^1_{\rm loc}\big(\R^d, L^1(\Omega)\big)$ is said to be stationary if
$$
\forall y \in \R^d, \quad a(x+y,\omega) = a(x,T_y\omega) \quad \mbox{for almost all $x\in \R^d$ and almost surely.}
$$
\end{definition}

We are now in position to state the following theorem (see~\cite[Theorem~12.4]{JKO}), which is a classical result in homogenization theory. 

\begin{theorem}
\label{homo-main}
Let $\A\in L^\infty\big(\R^d, L^1(\Omega,\mathcal{M})\big)$ be a tensor-valued stochastic field, stationary in the sense of Definition~\ref{De:stationary}. For all $\eps>0$, $x\in \R^d$ and $\omega \in \Omega$, set $\A^\eps(x,\omega) = \A(x/\eps,\omega)$. Then, the family $\left( \A^\eps(\cdot,\omega) \right)_{\eps>0}$ G-converges almost surely in the sense of Definition~\ref{Def-G} to a constant and deterministic tensor $A^\star \in \mathcal{M}$ in any bounded smooth domain $D\subset \R^d$. In addition, the tensor $A^\star$ is given by
\begin{equation}\label{eq:A*}
\forall \sigma \in \mathcal{S}, \qquad A^\star \sigma = \mathbb{E}\Big[\A(x,\cdot) \, \big(\sigma + e(w_\sigma(x,\cdot)) \big)\Big],
\end{equation}
where the above right-hand side actually does not depend on $x$ (by stationarity) and where $w_\sigma : \R^d \times \Omega \to \R^d$ is the solution (unique up to an additive constant vector) to the problem
\begin{equation}\label{eq:corrector}
\begin{cases}
- {\rm div} \Big[ \A(\cdot,\omega) \big( \sigma + e(w_\sigma(\cdot,\omega)) \big) \Big] = 0 \qquad \text{a.s. in $\R^d$}, \\
\text {$e(w_\sigma)$ is stationary in the sense of Definition~\ref{De:stationary}}, \\
\mathbb{E}\big[ e(w_\sigma(x,\cdot)) \big] = 0, 
\end{cases}
\end{equation}
where again, by stationarity, the left-hand side of the last line of~\eqref{eq:corrector} actually does not depend on $x$. 
\end{theorem}
We note that, starting in~\eqref{eq:A*}--\eqref{eq:corrector} and throughout this article, $\sigma$ denotes a constant symmetric matrix which is not to be confused with a stress tensor. 

\medskip

We end this section by another classical result of homogenization (see~\cite[Section~12.2]{JKO}) which will be useful in the sequel, and which considers a general family of G-convergent tensors.

\begin{theorem}
\label{jikov}
Let $D\subset \R^d$ and assume that $(\A^\eps)_{\eps>0} \subset L^\infty(D,\mathcal{M})$ is a family of tensor-valued fields which G-converges to $\A^\star \in L^\infty(D,\mathcal{M})$ in the domain $D$. Let $D_1\subset D$ be an arbitrary subdomain of $D$. For all $\sigma \in \mathcal{S}$ and $\eps>0$, let $w_\sigma^\eps\in H^1(D_1)^d$ be a solution to 
$$
-{\rm div} \big(\A^\eps (\sigma+e(w_\sigma^\eps)) \big) = 0 \quad \mbox{in $\mathcal{D}'(D_1)^d$}.
$$
If there exists $w_\sigma^\star\in H^1(D_1)^d$ such that $\dps w_\sigma^\eps \mathop{\rightharpoonup}_{\eps \to 0} w_\sigma^\star$ weakly in $H^1(D_1)^d$, then 
$$
-{\rm div} \big(\A^\star (\sigma+e(w_\sigma^\star)) \big) = 0 \quad \mbox{in $\mathcal{D}'(D_1)^d$}
$$
and 
$$
\A^\eps (\sigma+e(w_\sigma^\eps)) \mathop{\rightharpoonup}_{\eps\to 0} \A^\star (\sigma+e(w_\sigma^\star)) \qquad \text{weakly in $L^2(D_1)^{d\times d}$}. 
$$
\end{theorem}

\subsection{Numerical homogenization with periodic boundary conditions}

An important question of practical interest in the context of homogenization is to compute accurate approximations of the homogenized tensor $A^\star$. In the stochastic case, $A^\star$ is defined by~\eqref{eq:A*} in terms of the auxiliary functions $w_\sigma$ solution to the corrector problem~\eqref{eq:corrector}. The fact that this problem is posed over the whole space $\R^d$ is the key bottleneck in this context. Thus, in practice, one has to consider approximate problems, often defined on some truncated bounded domains with appropriate boundary conditions and compute approximations of $A^\star$ using the solution to these truncated corrector problems. 

\medskip

A classical strategy to approximate $A^\star$ is to consider truncated problems complemented with periodic boundary conditions. More precisely, for any $R>0$ and any $\sigma \in \mathcal{S}$, let  
$w^{R,{\rm per}}_\sigma(\cdot,\omega) \in H^1_{\rm per}\big((-R,R)^d\big)^d$ be the unique solution with vanishing mean of the following variational formulation:
$$
\forall v \in H^1_{\rm per}\big((-R,R)^d\big)^d, \quad \int_{(-R,R)^d} e(v) \cdot \A(\cdot,\omega) \left(\sigma + e \left(w_\sigma^{R,{\rm per}}(\cdot,\omega)\right) \right) = 0.
$$
The function $w_\sigma^{R,{\rm per}}(\cdot,\omega)$ in particular satisfies the following equation:
\begin{equation}\label{corrector-N}
- {\rm div} \Big(\A(\cdot,\omega) \left(\sigma + e \left(w_\sigma^{R,{\rm per}}(\cdot,\omega) \right) \right) \Big) = 0 \quad \text{a.s. in $(-R,R)^d$}.
\end{equation}
We next define the fourth-order tensor $A^{R,{\rm per}}(\omega) \in \cM$ by
\begin{equation} \label{eq:orthodon}
\forall \sigma \in \cS, \quad A^{R,{\rm per}}(\omega) \sigma = \frac{1}{(2R)^d} \int_{(-R,R)^d} \A(\cdot,\omega) \left(\sigma + e\left(w_\sigma^{R,{\rm per}}(\cdot,\omega) \right)\right).
\end{equation}
The following result is a direct extension of~\cite{Bourgeat}.

\begin{proposition}\label{prop:Bourgeat}
Assume that $\A\in L^\infty\big(\R^d, L^1(\Omega,\mathcal{M})\big)$ is a stationary tensor-valued stochastic field. Then the sequence $\left(A^{R,{\rm per}}(\omega)\right)_{R>0}$ defined by~\eqref{eq:orthodon} converges almost surely to $A^\star$ as $R$ goes to infinity.
\end{proposition}

Problem~\eqref{corrector-N} is usually approximated with standard finite element methods, which leads to discretized problems with a very large number of degrees of freedom since $R$ is large. Their resolution is thus computationally demanding. The purpose of this work is to introduce alternative approximations of the homogenized tensor, based on the use of a new approximate corrector problem, called hereafter the embedded problem and which is presented in details in the next section. As pointed out above, the motivation for considering these embedded problems is the following: in the special case of heterogeneous materials composed of isotropic spherical inclusions in an isotropic matrix, the proposed embedded corrector problem can be efficiently solved for large representative volume elements by a numerical method very close in spirit to the one proposed in~\cite{PART2} in the diffusion case. The numerical method for the elastic case is presented in full details in~\cite{SX}. 


\section{Embedded corrector problem}\label{sec:problem}

\subsection{Presentation of the embedded corrector problem}

We now present the alternative corrector problem we consider in this article, and on the basis of which we are going to provide new approximations of the homogenized tensor $A^\star$. This problem is similar to the one introduced in the case of scalar diffusion equations in~\cite{PART1,notre_cras}.

\medskip

For all $R>0$, let $B_R$ be the open ball of $\R^d$ centered at the origin and of radius $R$. Let also denote by $B = B_1$ the unit open ball and $\mathbb{S}$ the unit sphere. Let $n$ be the outward pointing normal vector on $\mathbb{S}$.

\medskip

For any $A \in \cM$ and $\A \in L^\infty(B,\cM)$, we define
\begin{equation}\label{bb-A}
\cA^{\A,A}(x) := 
\begin{cases}
\A(x) & \mbox{if $x \in B$}, \\
A & \mbox{if $x\in \R^d \setminus B$}. 
\end{cases}
\end{equation}
For any $\sigma \in \mathcal{S}$, consider the following problem, which we call hereafter an {\em embedded corrector problem}: find $w_\sigma^{\A,A} \in V_0$ solution to 
\begin{equation}\label{p-corrector}
- {\rm div} \left[ \mathcal{A}^{\A,A} \left(\sigma + e\left(w_\sigma^{\A,A}\right) \right) \right] = 0 \quad \text{in $\mathcal{D}'(\R^d)$},
\end{equation}
where the functional space $V_0$ is defined by
$$
V_0 = \left\{ v \in L_{\rm loc}^2(\R^d)^d, \quad \nabla v \in L^2(\R^d)^{d \times d}, \quad \int_B v = 0 \right\}.
$$
It can be easily seen that the space $V_0$ is a Hilbert space when endowed with the inner product $\langle \cdot,\cdot \rangle_{V_0}$ defined by
$$
\forall u,v\in V_0, \quad \langle u,v\rangle_{V_0} = \int_{\R^d} \nabla u \cdot \nabla v.  
$$
The variational form of problem~\eqref{p-corrector} reads as follows: find $w_\sigma^{\A,A} \in V_0$ such that
\begin{equation}\label{variation_1}
\forall v \in V_0, \quad \int_B e(v) \cdot \A \left(\sigma + e\left(w_\sigma^{\A,A}\right) \right) + \int_{\R^d \setminus B} e(v) \cdot A e\left(w_\sigma^{\A,A}\right) - \int_{\mathbb{S}} v \cdot ((A \sigma) \, n) = 0. 
\end{equation}
Equation~\eqref{variation_1} can be equivalently rewritten as
$$
\forall v\in V_0, \quad a^{\A,A}\left(w_\sigma^{\A,A},v\right) = b_\sigma^{\A,A}(v),
$$
where, for all $v,w\in V_0$, 
\begin{equation}\label{eq:linearforms}
a^{\A,A}(w,v) := \int_B e(v) \cdot \A e(w) + \int_{\R^d \setminus B} e(v) \cdot A e(w) \quad \mbox{and} \quad b_\sigma^{\A,A}(v) := - \int_B e(v) \cdot \A \sigma + \int_{\mathbb{S}} v \cdot ((A \sigma) \, n).
\end{equation}
In the sequel, for all $v\in V_0$, we denote by
\begin{align}\label{eq:E}
  \mathcal{E}_\sigma^{\A,A}(v)
  &=
  \frac{1}{|B|} \left( \int_B (\sigma + e(v)) \cdot \A (\sigma + e(v)) + \int_{\R^d\setminus B} e(v) \cdot A e(v) - 2 \int_{\mathbb{S}} v \cdot ((A \sigma) \, n) \right)
  \\
  \nonumber
  &=
  \frac{1}{|B|} \left( a^{\A,A}(v,v) - 2 b_\sigma^{\A,A}(v) + \int_B \sigma\cdot \A \sigma \right).
\end{align}
We then prove that there exists a unique solution to~\eqref{variation_1} (and equivalently to~\eqref{p-corrector}) in $V_0$. More precisely, we have the following propositions:

\begin{proposition}\label{prop:existence}
The bilinear form $a^{\A,A} : V_0 \times V_0\to \R$ is symmetric, continuous and coercive. The linear form $b_\sigma^{\A,A} : V_0 \to \R$ is continuous. There hence exists a unique solution $w_\sigma^{\A,A}$ in $V_0$ to~\eqref{variation_1}. Equivalently, $w_\sigma^{\A,A}$ is the unique solution to the minimization problem
\begin{equation}\label{u_chi}
  w_\sigma^{\A,A} = \mathop{\mbox{\rm argmin}}_{v\in V_0} \; \mathcal{E}_\sigma^{\A,A}(v).
\end{equation}
\end{proposition}

\begin{proposition}\label{prop:existence_edp}
The unique solution $w_\sigma^{\A,A}$ in $V_0$ to~\eqref{variation_1} is also the unique solution in $V_0$ to~\eqref{p-corrector}.
\end{proposition}

We stress the fact that, contrarily to the scalar diffusion case~\cite{PART1}, the proof of the coercivity of the bilinear form $a$ is not obvious in the elasticity case and is postponed (along with the proof of Proposition~\ref{prop:existence}) until Section~\ref{sec:proexis}. The proof of Proposition~\ref{prop:existence_edp} is itself postponed until Section~\ref{sec:proexis_edp}. 

\medskip

In the sequel, we denote by 
\begin{equation}\label{eq:Energy}
\mathcal{E}_\sigma^{\A}(A) := \mathcal{E}_\sigma^{\A,A}\left(w_\sigma^{\A,A}\right) = \inf_{v\in V_0} \mathcal{E}_\sigma^{\A,A}(v)
\end{equation}
and by
\begin{equation}\label{E-sum}
\mathcal{E}^{\A}(A) := \sum\limits_{1\leq i \leq j \leq d} \mathcal{E}_{\sigma^{ij}}^{\A}(A), 
\end{equation}
where $\sigma^{ij}$ is given by~\eqref{Sij} for all $1\leq i \leq j \leq d$. Using~\eqref{variation_1}, it holds that
\begin{align}
\mathcal{E}_\sigma^{\A}(A)
&= \mathcal{E}_\sigma^{\A,A}\left(w_\sigma^{\A,A}\right)
\nonumber
\\
&= \frac{1}{|B|} \left[ \int_B \sigma \cdot \A \sigma - \int_B e\left(w_\sigma^{\A,A}\right) \cdot \A e\left(w_\sigma^{\A,A}\right) - \int_{\R^d\setminus B} e\left(w_\sigma^{\A,A}\right) \cdot A e\left(w_\sigma^{\A,A}\right) \right]
\label{elseE_1}
\\
&= \frac{1}{|B|} \left[ \int_B \sigma\cdot \A \left(\sigma+e\left(w_\sigma^{\A,A}\right)\right) - \int_{\mathbb{S}} w_\sigma^{\A,A} \cdot ((A\sigma) \, n) \right].
\label{elseE_2}
\end{align} 

\subsection{Proof of Proposition~\ref{prop:existence}}\label{sec:proexis}

We start by proving several intermediate lemmas, and often make use of some auxiliary results related to Korn's inequalities which are collected in Appendix~\ref{sec:appendix}. We recall that $B_R$ is the ball of $\R^d$ centered at the origin and of radius $R$. 

\begin{lemma}\label{lem:use}
For all $u$ in $V_0$, we have $e(u) = 0$ in $\R^d$ if and only if $u=0$.
\end{lemma}

\begin{proof}
Let $R\geq1$. Let $u \in V_0$ with $e(u) = 0$. The function $u|_{B_R}$ belongs to $H^1(B_R)^d$. Since $e(u) = 0$, we can apply Lemma~\ref{lem:rigid} which yields that there exists a skew-symmetric matrix $M_R \in \R^{d\times d}$ and a vector $b_R \in \R^d$ such that $u(x) = M_R x + b_R$ for all $x \in B_R$. Using the fact that the mean of $u$ over $B \subset B_R$ vanishes, we obtain that $b_R = 0$. We thus have that, for any $R \geq 1$, $u(x) = M_R x$ for any $x\in B_R$. We infer that there exists a unique skew-symmetric matrix $M \in \R^{d\times d}$ such that $M_R = M$ for all $R>1$. We thus obtain $\nabla u = M$ in $\R^d$. Since $\nabla u \in L^2(\R^d)^{d\times d}$, we deduce that $M = 0$, and hence $u=0$.
\end{proof}

\begin{lemma}[Poincar\'e-Wirtinger inequality in $V_0$]\label{poincare2}
For all $R\geq 1$, there exists a constant $C(R)>0$ such that
\begin{equation}\label{eq:poincare}
\forall u \in V_0, \qquad \|u\|_{L^2(B_R)^d} \leq C(R) \, \|\nabla u\|_{L^2(B_R)^{d\times d}}.
\end{equation}
\end{lemma}

\begin{proof}
If $R= 1$, the classical Poincar\'e-Wirtinger inequality yields the desired result: there exists a constant $C_1>0$ such that, for all $u\in V_0$, 
$$
\left\| u \right\|_{L^2(B)^d} = \left\| u - \frac{1}{|B|} \int_B u \right\|_{L^2(B)^d} \leq C_1 \, \left\| \nabla u \right\|_{L^2(B)^{d\times d}}.
$$
Let now $R>1$ and let us argue by contradiction to prove~\eqref{eq:poincare}. We thus assume that there exists a sequence $(u_n)_{n\in\N^\star}$ of elements in $V_0$ such that 
$$
\left\|u_n\right\|_{L^2(B_R)^d} = 1, \qquad \left\|\nabla u_n \right\|_{L^2(B_R)^{d\times d}} \leq \frac{1}{n}. 
$$ 
Since $(u_n)_{n\in\N^\star}$ is bounded in $H^1(B_R)^d$, there exists $u_0\in H^1(B_R)^d$ such that, up to the extraction of a subsequence, $\dps u_n \mathop{\rightharpoonup}_{n\to +\infty} u_0$ weakly in $H^1(B_R)^d$. Hence, $\dps u_n \mathop{\rightarrow}_{n\to +\infty} u_0$ strongly in $L^2(B_R)^d$ and $\dps \nabla u_n \mathop{\rightharpoonup}_{n\to +\infty} \nabla u_0$ weakly in $L^2(B_R)^{d\times d}$. The above estimates on $u_n$ yield
$$
\|u_0\|_{L^2(B_R)^d} = 1, \qquad \left\| \nabla u_0 \right\|_{L^2(B_R)^{d\times d}} \leq \liminf_{n\to +\infty} \|\nabla u_n \|_{L^2(B_R)^{d\times d}} = 0. 
$$
Therefore, $u_0$ is equal to a constant vector on $B_R$. Moreover, using the Cauchy-Schwarz inequality, we have
$$
\left| \int_B u_n - \int_B u_0 \right| \leq |B|^{1/2} \ \|u_n -u_0\|_{L^2(B)^d} \leq |B|^{1/2} \ \|u_n -u_0\|_{L^2(B_R)^d} \mathop{\longrightarrow}_{n\to +\infty} 0.
$$  
The mean over $B$ of $u_0$ hence vanishes, and thus $u_0=0$, which contradicts the fact that $\|u_0\|_{L^2(B_R)^d} = 1$. This concludes the proof.
\end{proof}

The following lemma states that a Korn's inequality holds for functions in $V_0$.
\begin{lemma}[Korn's inequality in $V_0$] \label{Korn4}
For any $u$ in $V_0$, we have 
\begin{equation}\label{Kornall}
  \| \nabla u \|_{L^2(\R^d)^{d\times d}} \leq \sqrt{2} \ \| e(u) \|_{L^2(\R^d)^{d\times d}}.
\end{equation}
\end{lemma}

\begin{proof}
Let $u \in V_0$. For all $k\in\N^\star$, let us denote by $B_k \subset \R^d$ the ball centered at $0$ and of radius $k$. Using the Lax-Milgram lemma (and Lemma~\ref{Korn2} to show the coercivity of the associated bilinear form), it is easy to see that there exists a unique solution $\phi_k \in H_0^1(B_k)^d$ to
$$
- {\rm div} \big( e(\phi_k) \big) = - {\rm div} \big( e(u) \big) \quad \mbox{in $B_k$}.
$$
Since $\phi_k \in H_0^1(B_k)^d$, we can extend $\phi_k$ to a function in $H^1(\R^d)^d$ by setting $\phi_k = 0$ in $\R^d\setminus B_k$. Then, we have 
$$
\|e(\phi_k)\|_{L^2(B_k)^{d\times d}}^2 = \int_{B_k} e(\phi_k) \cdot e(\phi_k) = \int_{B_k} e(\phi_k) \cdot e(u) \leq \|e(\phi_k)\|_{L^2(B_k)^{d\times d}} \|e(u)\|_{L^2(B_k)^{d\times d}}
$$
and thus
$$
\| e(\phi_k) \|_{L^2(B_k)^{d\times d}} \leq \|e(u)\|_{L^2(B_k)^{d\times d}}.
$$
Using Lemma~\ref{Korn2}, we write
$$
\|\nabla \phi_k \|_{L^2(\R^d)^{d\times d}} = \|\nabla \phi_k \|_{L^2(B_k)^{d\times d}} \leq \sqrt{2} \, \|e(\phi_k)\|_{L^2(B_k)^{d\times d}} \leq \sqrt{2} \, \|e(u)\|_{L^2(B_k)^{d\times d}} \leq \sqrt{2} \, \|e(u)\|_{L^2(\R^d)^{d\times d}}.
$$
Consider now $\dps \psi_k = \phi_k - \int_B \phi_k$. Since $\|\nabla \psi_k \|_{L^2(\R^d)^{d\times d}} = \|\nabla \phi_k \|_{L^2(\R^d)^{d\times d}}$, we deduce from the above estimate that 
\begin{equation}\label{inequality-e-grad}
\|\nabla \psi_k \|_{L^2(\R^d)^{d\times d}} \leq \sqrt{2} \, \|e(u)\|_{L^2(\R^d)^{d\times d}},
\end{equation}
and the sequence $(\nabla \psi_k)_{k \in \N^\star}$ is thus bounded in $L^2(\R^d)^{d\times d}$. In addition, $\dps \int_B \psi_k = 0$ holds, and hence $\psi_k \in V_0$. For any $R>0$, Lemma~\ref{poincare2} provides the existence of a constant $C(R)>0$ such that  
$$
\forall k \in \N^\star, \qquad \left\|\psi_k \right\|_{H^1(B_R)^d} \leq C(R) \left\|\nabla \psi_k\right\|_{L^2(B_R)^{d\times d}}. 
$$
Collecting this bound with~\eqref{inequality-e-grad}, we obtain that $\psi_k$ is bounded in $H^1(B_R)^d$, and thus that there exists some $\psi_\infty^R \in H^1(B_R)^d$ such that (up to a subsequence extraction) $\dps \psi_k \mathop{\longrightarrow}_{k\to +\infty} \psi_\infty^R$ strongly in $L^2(B_R)^d$ and $\dps \nabla \psi_k \mathop{\rightharpoonup}_{k\to +\infty} \nabla \psi_\infty^R$ weakly in $L^2(B_R)^{d\times d}$. For all $R'>R>1$, we have $\psi_\infty^{R'}|_{B_R} = \psi_\infty^R$ so that the function $\psi_\infty \in H^1_{\rm loc}(\R^d)^d$ defined by $\psi_\infty(x) := \psi^R_\infty(x)$ for all $x \in B_R$ is well-defined. According to the definition of $\psi_\infty$, we have $\dps \psi_k \mathop{\longrightarrow}_{k\to +\infty} \psi_\infty$ in $L^2_{\rm loc}(\R^d)^d$ and $\dps \nabla \psi_k \mathop{\rightharpoonup}_{k\to +\infty} \nabla \psi_\infty$ in $L^2_{\rm loc}(\R^d)^{d\times d}$.

Let us show that $\psi_\infty \in V_0$. The estimate~\eqref{inequality-e-grad} yields that there exists $T_\infty \in L^2(\R^d)^{d\times d}$ such that, up to the extraction of a subsequence, $\dps \nabla \psi_k \mathop{\rightharpoonup}_{k\to +\infty} T_\infty$ in $L^2(\R^d)^{d\times d}$. The uniqueness of the weak limit of $\nabla \psi_k$ in $L^2_{\rm loc}(\R^d)^{d\times d}$ implies that $\nabla \psi_\infty = T_\infty$. In addition, we have 
$$
\left| \int_B \psi_k - \int_B \psi_\infty \right| \leq |B|^{1/2} \left\|\psi_k - \psi_\infty \right\|_{L^2(B)^d} \mathop{\longrightarrow}_{k\to +\infty} 0,
$$
and thus the mean of $\psi_\infty$ in $B$ vanishes. This implies that $\psi_\infty\in V_0$. Passing to the limit $k \to \infty$ in~\eqref{inequality-e-grad}, we also obtain
\begin{equation}\label{inequality-e-grad_bis}
  \|\nabla \psi_\infty \|_{L^2(\R^d)^{d\times d}} \leq \sqrt{2} \, \|e(u)\|_{L^2(\R^d)^{d\times d}}.
\end{equation}
 
\medskip 

The last step of the proof consists in proving that $\psi_\infty = u$. For any $\eta \in C_0^\infty(\R^d)^d$, let $N$ be such that ${\rm Supp} \; \eta \subset B_N$. Then, for all $k \geq N$, using the fact that $\eta \in H^1_0(B_k)^d$ and that $e(\psi_k) = e(\phi_k)$, we have that
$$
\left \langle -{\rm div} \, e(u), \eta \right\rangle_{\mathcal{D}'(\R^d)^d, \mathcal{D}(\R^d)^d} = \int_{B_N} e(\eta) \cdot e(u) = \int_{B_N} e(\eta) \cdot e(\phi_k) = \int_{B_N} e(\eta) \cdot e(\psi_k).
$$
Besides, it holds that 
$$
\lim_{k\to +\infty} \int_{B_N} e(\eta) \cdot e(\psi_k) = \int_{B_N} e(\eta) \cdot e(\psi_\infty) = \left\langle -{\rm div} \, e(\psi_\infty), \eta \right\rangle_{\mathcal{D}'(\R^d)^d, \mathcal{D}(\R^d)^d}.
$$
We hence get that
\begin{equation}\label{substraction}
- {\rm div} \big( e(u-\psi_\infty) \big) = 0 \quad \text{in $\mathcal{D}'(\R^d)^d$}. 
\end{equation}
For any $R>0$, let $\xi_R \in C^\infty_0(\R^d)$ satisfying $0\leq \xi_R \leq 1$,
\begin{equation} \label{eq:def_xi_R}
\xi_R(x)=\begin{cases}
1 & \mbox{ if $x \in B_R$}, \\
0 & \mbox{ if $x \notin B_{2R}$},
\end{cases}
\end{equation}
and $\left\| \nabla \xi_R \right\|_{L^\infty(B_{2R}\setminus B_R)^d} < C/R$ for some constant $C>0$ independent of $R$. Let us denote by $\dps a_R = \frac{1}{|B_{2R}\setminus B_R|} \int_{B_{2R}\setminus B_R} (u-\psi_\infty)$. Multiplying~\eqref{substraction} by $(u - \psi_\infty - a_R) \, \xi_R$ and integrating by parts yields
\begin{align*}
  0
  &=
  \int_{\R^d} e\big((u-\psi_\infty-a_R) \, \xi_R \big) \cdot e(u-\psi_\infty)
  \\
  &=
  \int_{B_R} e(u-\psi_\infty) \cdot e (u-\psi_\infty) + \int_{B_{2R} \setminus B_R} \xi_R \, e(u-\psi_\infty) \cdot e (u-\psi_\infty)
  \\
  & \qquad \qquad + \int_{B_{2R} \setminus B_R} \big[ \nabla \xi_R \otimes (u-\psi_\infty-a_R) \big] \cdot e(u-\psi_\infty),
\end{align*}
where $\otimes$ denotes the tensor product: for any vectors $v$ and $w$ and any $1 \leq i,j \leq d$, we have $(v \otimes w)_{ij} = v_i \, w_j$. We deduce from the above relation that
\begin{equation} \label{eq:covid}
\left\| e(u-\psi_\infty) \right\|^2_{L^2(B_R)^{d\times d}} \leq \left\| \nabla \xi_R \right\|_{L^\infty(\R^d)^d} \left\| e(u-\psi_\infty) \right\|_{L^2(B_{2R} \setminus B_R)^{d\times d}} \left\| u-\psi_\infty-a_R \right\|_{L^2(B_{2R} \setminus B_R)^d}. 
\end{equation}
Using the Poincar\'e-Wirtinger inequality together with a classical scaling argument, there exists a constant $C>0$ independent of $R$ such that
\begin{align*}
  \left\| u-\psi_\infty-a_R \right\|_{L^2(B_{2R} \setminus B_R)^d}
  &=
  \left\| u - \psi_\infty - \frac{1}{|B_{2R}\setminus B_R|} \int_{B_{2R}\setminus B_R} (u-\psi_\infty) \right\|_{L^2(B_{2R}\setminus B_R)^d}
  \\
  & \leq
  C \, R \left\| \nabla (u-\psi_\infty ) \right\|_{L^2(B_{2R}\setminus B_R)^{d\times d}}.
\end{align*}
Hence, using Lemma~\ref{simple}, we deduce from~\eqref{eq:covid} that
\begin{align*}
  \left\| e(u-\psi_\infty) \right\|^2_{L^2(B_R)^{d\times d}}
  & \leq
  C \left\| e(u-\psi_\infty) \right\|_{L^2(B_{2R} \setminus B_R)^{d\times d}} \left\| \nabla (u-\psi_\infty) \right\|_{L^2(B_{2R}\setminus B_R)^{d\times d}}
  \\
  & \leq
  C \left\| \nabla (u-\psi_\infty ) \right\|_{L^2(B_{2R}\setminus B_R)^{d\times d}}^2
\end{align*}
and thus
$$
\left\| e(u-\psi_\infty) \right\|_{L^2(B_R)^{d\times d}} \leq C \left\| \nabla (u-\psi_\infty) \right\|_{L^2(B_{2R}\setminus B_R)^{d\times d}}.
$$
Since $\nabla (u-\psi_\infty) \in L^2(\R^d)^{d\times d}$, we have $\dps \left\| \nabla (u-\psi_\infty) \right\|_{L^2(B_{2R}\setminus B_R)^{d\times d}} \mathop{\longrightarrow}_{R\to +\infty} 0$. Passing to the limit $R \to \infty$ in the above estimate, we therefore obtain $\left\| e(u-\psi_\infty) \right\|_{L^2(\R^d)^{d\times d}}=0$. Using that $u-\psi_\infty \in V_0$ and Lemma~\ref{lem:use}, this implies that $u=\psi_\infty$. Inserting this in~\eqref{inequality-e-grad_bis}, we deduce~\eqref{Kornall}. This concludes the proof of Lemma~\ref{Korn4}.
\end{proof}

We are now in position to prove Proposition~\ref{prop:existence}. 

\begin{proof}[Proof of Proposition~\ref{prop:existence}]
The bilinear form $a^{\A,A}$ defined by~\eqref{eq:linearforms} is obviously continuous in $V_0$. It is coercive in $V_0$ in view of Lemma~\ref{Korn4}. The linear form $b_\sigma^{\A,A}$ defined by~\eqref{eq:linearforms} is continuous in $V_0$, using the continuity of the trace operator from $H^1(B)^d$ to $L^2(\mathbb{S})^d$. We are thus in position to apply the Lax-Milgram lemma, which yields the existence and uniqueness of a solution to~\eqref{variation_1} in $V_0$. Since the bilinear form $a^{\A,A}$ is symmetric, the variational formulation~\eqref{variation_1} is equivalent to the energetic formulation~\eqref{u_chi}.
\end{proof}

\subsection{Proof of Proposition~\ref{prop:existence_edp}} \label{sec:proexis_edp}

We first prove that the unique solution $w_\sigma^{\A,A}$ in $V_0$ to~\eqref{variation_1} is also a solution to~\eqref{p-corrector}. Considering some $\varphi \in C^\infty_0(\R^d)$, we introduce $v = \varphi - m_\varphi$ with $\dps m_\varphi = \frac{1}{|B|} \int_B \varphi$, which indeed belongs to $V_0$. We then have
$$
a^{\A,A}\left(w_\sigma^{\A,A},\varphi\right) = a^{\A,A}\left(w_\sigma^{\A,A},v\right) = b_\sigma^{\A,A}(v) = b_\sigma^{\A,A}(\varphi) - \int_{\mathbb{S}} m_\varphi \cdot ((A \sigma) \, n),
$$
and the last term vanishes because $m_\varphi$, $A$ and $\sigma$ are constant on $\mathbb{S}$. We thus get that $\dps a^{\A,A}\left(w_\sigma^{\A,A},\varphi\right) = b_\sigma^{\A,A}(\varphi)$ for any $\varphi \in C^\infty_0(\R^d)$, which implies that $w_\sigma^{\A,A}$ is a solution to~\eqref{p-corrector}.

\medskip

We next prove that~\eqref{p-corrector} has a unique solution in $V_0$. To that aim, it is sufficient to prove that, if $w \in V_0$ is such that
\begin{equation}\label{p-corrector_homogene}
- {\rm div} \left( \mathcal{A}^{\A,A} \, e(w) \right) = 0 \quad \text{in $\mathcal{D}'(\R^d)^d$},
\end{equation}
then $w=0$. We essentially use the same arguments as those below~\eqref{substraction} and again introduce the function $\xi_R$ defined by~\eqref{eq:def_xi_R}. We multiply~\eqref{p-corrector_homogene} by $v = \xi_R^2 \, \big( w - m_R(w) \big)$, where $m_R(w)$ is the mean of $w$ on $B_{2R} \setminus B_R$, and deduce that
$$
0 = \int_{B_{2R}} \xi_R^2 \ e(w) \cdot \mathcal{A}^{\A,A} e(w) + 2 \int_{B_{2R}} \xi_R \, \left[ \nabla \xi_R \otimes \big(w - m_R(w)\big) \right] \cdot \mathcal{A}^{\A,A} e(w).
$$
Using the Cauchy-Schwarz inequality and next~\eqref{eq:ellipticity}, we obtain
\begin{align*}
\int_{B_{2R}} \xi_R^2 \ e(w) \cdot \mathcal{A}^{\A,A} e(w)
&\leq
4 \int_{B_{2R}} \left[ \nabla \xi_R \otimes \big(w - m_R(w)\big) \right] \cdot \mathcal{A}^{\A,A} \left[ \nabla \xi_R \otimes \big(w - m_R(w)\big) \right]
\\
&\leq
\frac{4 C \beta}{R^2} \int_{B_{2R} \setminus B_R} \big| w - m_R(w) \big|^2.
\end{align*}
We next bound the above right-hand side using the Poincar\'e-Wirtinger inequality on $B_{2R} \setminus B_R$, and we bound the left-hand side using~\eqref{eq:ellipticity} and the properties of $\xi_R$. We thus deduce that, for any $R$,
$$
\| e(w) \|^2_{L^2(B_R)^{d \times d}} \leq C \, \| \nabla w \|^2_{L^2(B_{2R} \setminus B_R)^{d \times d}}.
$$
Using that $\nabla w$ belongs to $L^2(\R^d)^{d \times d}$, we are in position to pass to the limit $R \to \infty$ in the above estimate, which yields that $\| e(w) \|^2_{L^2(\R^d)^{d \times d}} = 0$. Using Lemma~\ref{lem:use}, we deduce that $w=0$, which concludes the proof. In passing, note that we do not have used the fact that $\mathcal{A}^{\A,A}$ is constant on $\R^d \setminus \overline{B}$, but only the fact that $\mathcal{A}^{\A,A} \in L^\infty(\R^d,\cM)$.

\subsection{A useful auxiliary result}

We now show the following result, which will be useful in the sequel. We recall that $B$ is the unit open ball of $\R^d$ centered at the origin.

\begin{lemma} \label{lem:uniqueness-neumann}
  Let $A \in \cM$ and let $\dps w \in \left\{ v \in L_{\rm loc}^2(\R^d \setminus \overline{B})^d, \ \ \nabla v \in L^2(\R^d \setminus \overline{B})^{d \times d} \right\}$ such that $- {\rm div} \left( A \, e(w) \right) = 0$ in $\mathcal{D}'(\R^d \setminus \overline{B})^d$ and $(A \, e(w)) \, n = 0$ on $\mathbb{S}$. Then $w$ is a constant on $\R^d \setminus \overline{B}$.
\end{lemma}

The proof below does not use the fact that $A$ is constant on $\R^d \setminus \overline{B}$. This result could thus be generalized (although we do not need this generalization in this article) to fourth-order tensor fields valued in $\cM$. 

\begin{proof}
We again follow the arguments of the proof of Proposition~\ref{prop:existence_edp}: for any $R \geq 1$, we multiply the equation by $v = \xi_R^2 \, \big( w - m_R(w) \big)$, where $m_R(w)$ is the mean of $w$ on $B_{2R} \setminus B_R$ and where $\xi_R$ is defined by~\eqref{eq:def_xi_R}. We thus have
$$
0
=
- \int_{B_{2R} \setminus B} \xi_R^2 \, \big( w - m_R(w) \big) \, {\rm div} \left( A \, e(w) \right)
=
\int_{B_{2R} \setminus B} e\left( \xi_R^2 \, \big( w - m_R(w) \big) \right) \cdot A \, e(w)
$$
where the boundary term in the integration by parts vanishes since $(A \, e(w)) \, n = 0$ on $\mathbb{S}$. We thus deduce that
$$
0 = \int_{B_{2R} \setminus B} \xi_R^2 \ e(w) \cdot A e(w) + 2 \int_{B_{2R} \setminus B} \xi_R \, \left[ \nabla \xi_R \otimes \big(w - m_R(w)\big) \right] \cdot A e(w),
$$
and therefore, using the Cauchy-Schwarz inequality and the fact that $A \in \cM$, that
\begin{align*}
\int_{B_{2R} \setminus B} \xi_R^2 \ e(w) \cdot A e(w)
&\leq
4 \int_{B_{2R} \setminus B} \left[ \nabla \xi_R \otimes \big(w - m_R(w)\big) \right] \cdot A \left[ \nabla \xi_R \otimes \big(w - m_R(w)\big) \right]
\\
&\leq
\frac{4 C \beta}{R^2} \int_{B_{2R} \setminus B_R} \big| w - m_R(w) \big|^2.
\end{align*}
We bound the above right-hand side using the Poincar\'e-Wirtinger inequality on $B_{2R} \setminus B_R$ and the above left-hand side using~\eqref{eq:ellipticity} and the properties of $\xi_R$. We thus deduce that, for any $R \geq 1$,
$$
\| e(w) \|^2_{L^2(B_R \setminus \overline{B})^{d \times d}} \leq C \, \| \nabla w \|^2_{L^2(B_{2R} \setminus B_R)^{d \times d}}.
$$
Using that $\nabla w$ belongs to $L^2(\R^d \setminus \overline{B})^{d \times d}$, we are in position to pass to the limit $R \to \infty$ in the above estimate, which yields that $\| e(w) \|^2_{L^2(\R^d \setminus \overline{B})^{d \times d}} = 0$, and hence $e(w) = 0$ in $\R^d \setminus \overline{B}$. For any $R \geq 1$, we thus have that $w$ is a rigid displacement of $B_R \setminus \overline{B}$, and hence, using Lemma~\ref{lem:rigid}, that $w(x) = M_R x + b_R$ in $B_R \setminus \overline{B}$. This implies that $\nabla w = M_R$ in $B_R \setminus \overline{B}$, for any $R \geq 1$. We thus infer that there exists a unique skew-symmetric matrix $M$ such that $M_R = M$, and hence $\nabla w = M$ in $\R^d \setminus \overline{B}$. Since $\nabla w \in L^2(\R^d \setminus \overline{B})^{d \times d}$, we deduce that $M=0$. We thus get that $w(x) = b_R$ in $B_R \setminus \overline{B}$ for any $R \geq 1$, which concludes the proof. 
\end{proof}

\subsection{Properties of the embedded corrector problem}

In this section, we prove some auxiliary results on the solutions to problems of the form~\eqref{p-corrector}. These results will be useful in our analysis below and are direct extensions of similar results in the scalar case (see~\cite{PART1}). We start with Lemma~\ref{converge-in-out}, which is a direct extension of~\cite[Lemma~3.1]{PART1}. 

\begin{lemma}
\label{converge-in-out}
Let $(\A^N)_{N \in \N}$ and $(A^N)_{N \in \N}$ be two sequences such that, for any $N$, $\A^N \in L^\infty(B,\mathcal{M})$ and $A^N \in \mathcal{M}$. We assume that there exist $\A^\star \in L^\infty(B,\mathcal{M})$ and $A \in \mathcal{M}$ such that $(\A^N)_{N \in \N}$ converges as $N$ goes to infinity in the sense of homogenization to $\A^\star$ and such that $\dps A^N \mathop{\longrightarrow}_{N\to +\infty} A$. Let $w_\sigma^{\A^N,A^N} \in V_0$ be the unique solution to the problem
\begin{equation}\label{pro-1}
- {\rm div} \left[ \cA^{\A^N,A^N} \left( \sigma + e\left(w_\sigma^{\A^N,A^N}\right) \right) \right] = 0 \quad \text{in $\mathcal{D}'(\R^d)^d$}. 
\end{equation}
Then, $\left(w_\sigma^{\A^N,A^N}\right)_{N \in \N}$ weakly converges in $H^1_{\rm loc}(\R^d)^d$ to $w_\sigma^{\A^\star,A}$, which is the unique solution in $V_0$ to
$$
- {\rm div} \left[ \cA^{\A^\star,A} \left( \sigma +e\left(w_\sigma^{\A^\star,A} \right) \right) \right] = 0 \quad \mbox{in $\mathcal{D}'(\R^d)^d$},
$$
where $\cA^{\A^\star,A}(x) = \A^\star(x)$ in $B$ and $\cA^{\A^\star,A}(x) = A$ in $\R^d \setminus B$. Moreover, we have 
$$
\cA^{\A^N,A^N} \left( \sigma + e\left(w_\sigma^{\A^N,A^N}\right) \right) \mathop{\rightharpoonup}_{N \to +\infty} \cA^{\A^\star,A} \left( \sigma + e\left(w_\sigma^{\A^\star,A}\right) \right) \quad \mbox{weakly in $L^2_{\rm loc}(\R^d)^{d\times d}$}. 
$$
\end{lemma}

\begin{proof}
Using the variational formulation of~\eqref{pro-1}, and taking $w_\sigma^{\A^N,A^N}$ as a test function, we obtain
\begin{align*}
  \alpha \left\| e\left(w_\sigma^{\A^N,A^N}\right) \right\|^2_{L^2(\R^d)^{d\times d}}
  & \leq
  \int_{\R^d} e(w_\sigma^{\A^N,A^N}) \cdot \cA^{\A^N,A^N} e\left(w_\sigma^{\A^N,A^N}\right)
  \\
  & =
  - \int_B e\left(w_\sigma^{\A^N,A^N}\right) \cdot \A^N \sigma + \int_{\mathbb{S}} w_\sigma^{\A^N,A^N} \cdot ((A^N \sigma) \, n)
  \\
  & \leq
  \beta \left\| e\left(w_\sigma^{\A^N,A^N}\right) \right\|_{L^2(B)^{d\times d}} \left\| \sigma \right\|_{L^2(B)^{d\times d}} + \left\| A^N \sigma \right\|_{L^2(\mathbb{S})^{d\times d}} \left\| w_\sigma^{\A^N,A^N} \right\|_{L^2(\mathbb{S})^d}.
\end{align*}
Using the continuity of the trace application from $H^1(B)^d$ to $L^2(\mathbb{S})^d$ together with the Poincar\'e-Wirtinger inequality in $B$, there exists a constant $C>0$ such that $\left\| w_\sigma^{\A^R,A^R} \right\|_{L^2(\mathbb{S})^d} \leq C \left\| \nabla w_\sigma^{\A^N,A^N} \right\|_{L^2(B)^{d\times d}}$. Using additionally Lemma~\ref{simple}, we deduce from the above bound that
$$
\alpha \left\| e\left(w_\sigma^{\A^N,A^N}\right) \right\|^2_{L^2(\R^d)^{d\times d}}
\leq
\beta \left\| \nabla w_\sigma^{\A^N,A^N} \right\|_{L^2(B)^{d\times d}} \left\| \sigma \right\|_{L^2(B)^{d\times d}} + C \left\| A^N \sigma \right\|_{L^2(\mathbb{S})^{d\times d}} \left\| \nabla w_\sigma^{\A^N,A^N} \right\|_{L^2(B)^{d \times d}}.
$$
Using now Lemma~\ref{Korn4}, we deduce that the sequence $\left( \nabla w_\sigma^{\A^N,A^N} \right)_{N \in \N}$ is bounded in $L^2(\R^d)^{d\times d}$. 

\medskip

Let $R>1$. Using Lemma~\ref{poincare2}, there exists a constant $C(R)>0$ such that $\left\| w_\sigma^{\A^N,A^N} \right\|_{L^2(B_R)^d} \leq C(R) \left\| \nabla w_\sigma^{\A^N,A^N} \right\|_{L^2(\R^d)^{d\times d}}$. The sequence $\left( w_\sigma^{\A^N,A^N} \right)_{N \in \N}$ is thus bounded in $H^1(B_R)^d$, and hence there exists $w_{\sigma,R}^\infty \in H^1(B_R)^d$ such that 
$$
w_\sigma^{\A^N,A^N} \mathop{\rightharpoonup}_{N \to \infty} w_{\sigma,R}^\infty \quad \mbox{weakly in $H^1(B_R)^d$}. 
$$
The function $w_\sigma^\infty \in H_{\rm loc}^1(\R^d)^d$ defined by $w_\sigma^\infty(x) := w_{\sigma, R}^\infty(x)$ for $x \in B_R$ is well-defined since $w_{\sigma,R'}^\infty|_{B_R} = w_{\sigma,R}^\infty$ for all $R'>R$. The sequence $\left( w_\sigma^{\A^N,A^N} \right)_{N \in \N}$ thus weakly converges in $H_{\rm loc}^1(\R^d)^d$ to $w_\sigma^\infty$. Moreover, using the fact that $\left( \nabla w_\sigma^{\A^N,A^N} \right)_{N \in \N}$ is bounded in $L^2(\R^d)^{d\times d}$, we obtain that $\nabla w_\sigma^\infty \in L^2(\R^d)^{d\times d}$, and thus $w_\sigma^\infty \in V_0$. Applying Theorem~\ref{jikov} on the unit ball $B$, we obtain that 
$$
\A^N \left( \sigma+e\left(w_\sigma^{\A^N,A^N}\right) \right) \mathop{\rightharpoonup}_{N \to \infty} \A^\star \left( \sigma+e\left(w_\sigma^\infty\right) \right) \quad \mbox{weakly in $L^2(B)^{d\times d}$}. 
$$
Similarly, for any compact domain $K \subset \R^d \setminus B$, we have 
$$
A^N \left(\sigma+e\left(w_\sigma^{\A^N,A^N}\right) \right) \mathop{\rightharpoonup}_{N \to \infty} A \left(\sigma+e\left(w_\sigma^\infty\right) \right) \quad \mbox{weakly in $L^2(K)^{d\times d}$}. 
$$
This implies that the sequence $\left(\cA^{\A^N,A^N} \left(\sigma+e\left(w_\sigma^{\A^N,A^N}\right) \right) \right)_{N \in \N}$ weakly converges to $\cA^{\A^\star,A} \left(\sigma+e\left(w_\sigma^\infty\right) \right)$ in $L^2_{\rm loc}(\R^d)^{d\times d}$. Using~\eqref{pro-1}, we deduce that
$$
- {\rm div} \left( \cA^{\A^\star,A} \left(\sigma+e\left(w_\sigma^\infty\right) \right) \right) = 0 \quad \mbox{in $\mathcal{D}'(\R^d)^d$}.
$$
The above equation has a unique solution in $V_0$, and $w_\sigma^\infty \in V_0$. We thus obtain $w_\sigma^\infty = w_\sigma^{\A^\star,A}$, which concludes the proof of Lemma~\ref{converge-in-out}.
\end{proof}

We next turn to Lemma~\ref{same-tensor}, which is similar to~\cite[Lemma~3.2]{PART1}. 

\begin{lemma}
\label{same-tensor}
Let $A, A^\star\in \cM$. For all $\sigma \in \cS$, let $w_\sigma^{A^\star,A} \in V_0$ be the unique solution to 
\begin{equation}\label{corrector_bis}
- {\rm div} \left[ \cA^{A^\star,A} \left(\sigma+e\left(w_\sigma^{A^\star,A} \right)\right) \right] = 0 \quad \mbox{in $\mathcal{D}'(\R^d)^d$},
\end{equation}
where
$$
\cA^{A^\star,A}= 
\begin{cases}
A^\star & \mbox{if $x \in B$}, \\
A & \mbox{if $x\in \R^d \setminus B$}.
\end{cases}
$$
Then, $A^\star = A$ if and only if, for any $\sigma \in \cS$, $w_\sigma^{A^\star,A} =0$. 
\end{lemma}

\begin{proof}
If $A^\star = A$, then $w_\sigma^{A^\star,A} = 0$ is indeed the unique solution in $V_0$ to~\eqref{corrector_bis} for any $\sigma \in \cS$. 

Let us prove the converse and assume that, for any $\sigma \in \cS$, $w_\sigma^{A^\star,A} = 0$. Then, for all $v \in \mathcal{D}(\R^d)^d$, we have $\dps \int_{\R^d} e(v) \cdot \cA^{A^\star,A} \sigma = 0$. This implies that $\dps \int_{\mathbb{S}} \left(\left(A^\star \sigma\right) \, n \right) \cdot v = \int_{\mathbb{S}} \left(\left(A \sigma\right) \, n \right) \cdot v$. Since $v$ is arbitrary, we obtain $\left(A^\star\sigma\right) \, n = \left(A \sigma\right) \, n$ on $\mathbb{S}$, and hence $A^\star\sigma = A\sigma$ for all $\sigma\in\cS$ (since $A$ and $A^\star$ are constant). Since $A$ and $A^\star$ are symmetric tensors, we deduce $A^\star = A$.
\end{proof}

\begin{lemma}
\label{concave}
For all $\sigma\in \cS$ and $\A \in L^\infty(B,\mathcal{M})$, the mapping $\cM \ni A \mapsto \mathcal{E}_\sigma^{\A}(A)$ is concave (recall that $\mathcal{E}_\sigma^{\A}(A)$ is defined by~\eqref{eq:Energy}). In other words, for any $\alpha \in [0,1]$ and $A_0, A_1 \in \mathcal{M}$, we have
$$
\mathcal{E}_\sigma^{\A}(\alpha A_1+(1-\alpha)A_0) \geq \alpha \mathcal{E}_\sigma^{\A}(A_1) + (1-\alpha) \mathcal{E}_\sigma^{\A}(A_0). 
$$
In addition, we have the following property. Let $\alpha \in (0,1)$. The equality $\mathcal{E}_\sigma^{\A}(\alpha A_1+(1-\alpha)A_0) = \alpha \mathcal{E}_\sigma^{\A}(A_1) + (1-\alpha) \mathcal{E}_\sigma^{\A}(A_0)$ holds if and only if $w_\sigma^{\A,\alpha A_1+(1-\alpha)A_0} = w_\sigma^{\A,A_1} = w_\sigma^{\A,A_0}$.
\end{lemma}

\begin{proof}
Denote by $A_\alpha = \alpha A_1 + (1-\alpha) A_0$ for all $\alpha\in[0,1]$. Since the mapping $\cM \ni A \mapsto \mathcal{E}_\sigma^{\A,A}(v)$ is affine for any fixed $v\in V_0$, we have
\begin{align*}
  \mathcal{E}_\sigma^{\A}(A_\alpha)
  & =
  \mathcal{E}_\sigma^{\A,A_\alpha} \left( w_\sigma^{\A,A_\alpha} \right)
  \\
  & =
  \alpha \mathcal{E}_\sigma^{\A,A_1} \left(w_\sigma^{\A,A_\alpha}\right) + (1-\alpha) \mathcal{E}_\sigma^{\A,A_0} \left(w_\sigma^{\A,A_\alpha}\right)
  \\ 
  & \geq
  \alpha \mathcal{E}_\sigma^{\A,A_1} \left(w_\sigma^{\A,A_1}\right) + (1-\alpha) \mathcal{E}_\sigma^{\A,A_0} \left(w_\sigma^{\A,A_0}\right)
  \\
  &= 
  \alpha \mathcal{E}_\sigma^{\A}(A_1) + (1-\alpha) \mathcal{E}_\sigma^{\A}(A_0).
\end{align*}
We have thus established the concavity of the mapping $\cM \ni A \mapsto \mathcal{E}_\sigma^{\A}(A)$.

Assume now that $A_0$ and $A_1$ are such that $\mathcal{E}_\sigma^{\A}(A_\alpha) = \alpha \mathcal{E}_\sigma^{\A}(A_1) + (1-\alpha) \mathcal{E}_\sigma^{\A}(A_0)$ for some $\alpha\in[0,1]$. All the inequalities above should then be equalities, and we thus deduce that 
$$
\mathcal{E}_\sigma^{\A,A_1} \left(w_\sigma^{\A,A_\alpha} \right) = \mathcal{E}_\sigma^{\A,A_1} \left(w_\sigma^{\A,A_1} \right) \quad \mbox{and} \quad \mathcal{E}_\sigma^{\A,A_0} \left(w_\sigma^{\A,A_\alpha} \right) = \mathcal{E}_\sigma^{\A,A_0} \left(w_\sigma^{\A,A_0} \right).
$$
Since $w_\sigma^{\A,A_1}$ (resp. $w_\sigma^{\A,A_0}$) is the unique minimizer of $\mathcal{E}_\sigma^{A,A_1}$ (resp. $\mathcal{E}_\sigma^{A,A_0}$), we deduce that $w_\sigma^{\A,A_1} = w_\sigma^{\A,A_\alpha} = w_\sigma^{\A,A_0}$. The converse statement obviously holds true.
\end{proof}


\section{Alternative approximations of the homogenized tensor} \label{sec:Alternative}

In this section, we consider $\left(\A^N\right)_{N \in \N} \subset L^\infty(B,\cM)$ a family of tensor-valued functions which is assumed to converge in $B$ in the sense of homogenization to a {\em constant} tensor $A^\star\in \cM$. As mentioned in the introduction, stochastic homogenization is a prototypical example of context where this situation occurs and where computing effective approximations of the homogenized tensor $A^\star$ is of practical importance. 

The aim of this section is to propose four alternative definitions of effective tensors, which all make use of embedded corrector problems of the form~\eqref{p-corrector}, and to prove their convergence to the homogenized tensor $A^\star$ when $N$ goes to infinity. These approximations are inspired by our earlier work~\cite{PART1} on a simple scalar diffusion problem.

We recall (see beginning of Section~\ref{sec:pre}) that $\widetilde{\cM}$ is the set of symmetric fourth-order tensors, and that $\cM \subset \widetilde{\cM}$ is the set of tensors additionally satisfying the ellipticity and boundedness assumptions~\eqref{eq:ellipticity}.

\medskip

\noindent
{\bf Approximation 1:} Let $A_1^N \in \cM$ be a tensor satisfying  
\begin{equation}\label{tensor1}
A_1^N \in \underset{A \in \mathcal{M}}{\mathop{\rm argmax}} \ \mathcal{E}^{\A^N}(A), 
\end{equation}
where $\cM \ni A \mapsto \mathcal{E}^{\A^N}(A)$ is defined by~\eqref{E-sum}. Since this mapping is concave (in view of Lemma~\ref{concave}), the existence of a tensor $A_1^N$ in $\cM$ satisfying~\eqref{tensor1} is obvious.

\medskip

\noindent
{\bf Approximation 2:} Our second approximation $A_2^N\in \widetilde{\cM}$ is defined by
\begin{equation}\label{tensor2}
\forall \sigma \in \cS, \qquad A_2^N \sigma = \frac{1}{|B|} \int_B \A^N \left(\sigma+ e\left(w_\sigma^{\A^N ,A_1^N}\right)\right).
\end{equation}

\medskip

\noindent
{\bf Approximation 3:} Observing that $\mathcal{E}_\sigma^{\A^N}(A^N_1)$ quadratically depends on $\sigma$, there exists a unique $A^N_3 \in \widetilde{\cM}$ such that 
\begin{equation}\label{tensor3}
\forall \sigma \in \cS, \qquad \sigma \cdot A^N_3 \sigma = \mathcal{E}_\sigma^{\A^N}(A^N_1).
\end{equation}

\medskip

\noindent
{\bf Approximation 4:} Finally, we consider a self-consistent formulation to define our fourth approximation. Let us assume that, for all $N \in \N$, there exists $A^N_4 \in \mathcal{M}$ satisfying
\begin{equation}\label{tensor4}
\forall \sigma \in \cS, \qquad \sigma \cdot A_4^N \sigma = \mathcal{E}_\sigma^{\A^N}(A_4^N).
\end{equation}
We then have the following proposition, which is a direct extension of~\cite[Propositions~3.4 and~3.5]{PART1} to the elasticity case. 
 
\begin{proposition}\label{prop:convergence}
Let $(\A^N)_{N \in \N} \subset L^\infty(B,\mathcal{M})$ a family of tensors which converges in $B$ in the sense of homogenization to a constant tensor $A^\star\in \cM$. Let $A_1^N$, $A_2^N$ and $A_3^N$ be respectively defined by~\eqref{tensor1}, \eqref{tensor2} and~\eqref{tensor3}. Then, we have 
$$
A_1^N \mathop{\longrightarrow}_{N\to +\infty} A^\star, \qquad A_2^N \mathop{\longrightarrow}_{N\to +\infty} A^\star \qquad \mbox{and} \qquad A_3^N \mathop{\longrightarrow}_{N\to +\infty} A^\star.  
$$
In addition, let us assume that there exists a family $(A_4^N)_{N \in \N} \subset \cM$ satisfying~\eqref{tensor4}. Then, 
$$
A_4^N \mathop{\longrightarrow}_{N\to +\infty} A^\star.
$$
\end{proposition}

\begin{proof}
We sucessively study each of our four approximations.

\medskip

\noindent
{\bf Approximation 1:} For all $N \in \N$, we have $A_1^N \in \mathcal{M}$ and $\left( A_1^N \right)_{N \in \N}$ is thus bounded. There thus exists $A_1^\infty \in \cM$ such that (up to the extraction of a subsequence) $\dps A_1^N \mathop{\longrightarrow}_{N\to +\infty} A_1^\infty$. Let us prove that $A_1^\infty = A^\star$.

For all $1\leq i\leq j \leq d$, consider the unique solutions $w_{\sigma^{ij}}^{\A^N,A^\star}$ and $w_{\sigma^{ij}}^{A^\star,A^\star}$ in $V_0$ to the problems
\begin{gather*}
  - {\rm div} \left[ \cA^{\A^N,A^\star} \left(\sigma^{ij}+e\left(w_{\sigma^{ij}}^{\A^N,A^\star}\right)\right) \right] = 0 \quad \text{in $\mathcal{D}'(\R^d)^d$},
  \\
  - {\rm div} \left[ \cA^{A^\star,A^\star} \left(\sigma^{ij}+e\left(w_{\sigma^{ij}}^{A^\star,A^\star}\right)\right) \right] = 0 \quad \text{in $\mathcal{D}'(\R^d)^d$}. 
\end{gather*}
Using Lemma~\ref{same-tensor}, it holds that $w_{\sigma^{ij}}^{A^\star,A^\star}=0$. Lemma~\ref{converge-in-out} implies that
\begin{equation}\label{eq:wcv1}
w_{\sigma^{ij}}^{\A^N,A^\star} \xrightharpoonup[N\to \infty]{} w_{\sigma^{ij}}^{A^\star,A^\star} = 0 \qquad \text{weakly in $H^1_{\rm loc}(\R^d)^d$},
\end{equation}
and
\begin{equation}\label{eq:wcv2}
 \cA^{\A^N,A^\star} \left(\sigma^{ij} + e\left(w_{\sigma^{ij}}^{\A^N,A^\star}\right)\right) \xrightharpoonup[N\to \infty]{} A^\star \left(\sigma^{ij} + e\left(w_{\sigma_{ij}}^{A^\star,A^\star}\right)\right) = A^\star \sigma^{ij} \qquad \text{weakly in $L^2_{\rm loc}(\R^d)^{d\times d}$}. 
\end{equation}
Since $A_1^N$ is a maximizer of~\eqref{tensor1} and since $A^\star \in \cM$, we have 
\begin{equation} \label{eq:covid3}
\sum\limits_{1\leq i\leq j\leq d} \mathcal{E}_{\sigma^{ij}}^{\A^N}(A_1^N) \geq \sum\limits_{1\leq i\leq j \leq d} \mathcal{E}_{\sigma^{ij}}^{\A^N}(A^\star).
\end{equation}
We are going to pass to the limit $N \to \infty$ in the above estimate. Using~\eqref{elseE_2}, we write
\begin{equation} \label{eq:covid2}
  |B| \, \mathcal{E}_{\sigma^{ij}}^{\A^N}(A^\star)
  =
  \int_B \sigma^{ij} \cdot \A^N \left(\sigma^{ij}+e\left(w^{\A^N,A^\star}_{\sigma^{ij}}\right) \right) - \int_{\mathbb{S}} w^{\A^N,A^\star}_{\sigma^{ij}} \cdot \left(\left(A^\star\sigma^{ij}\right) \, n\right).
\end{equation}
To pass to the limit in the first term of the right-hand side of~\eqref{eq:covid2}, we use~\eqref{eq:wcv2}, which yields
$$
\int_B \sigma^{ij} \cdot \A^N \left(\sigma^{ij}+e\left(w^{\A^N,A^\star}_{\sigma^{ij}}\right) \right) \mathop{\longrightarrow}_{N\to +\infty} \int_B \sigma^{ij} \cdot A^\star \sigma^{ij}. 
$$ 
Since the trace operator is compact from $H^1(B)$ to $L^2(\mathbb{S})$, we deduce from~\eqref{eq:wcv1} that $\dps w^{\A^N,A^\star}_{\sigma^{ij}}\mathop{\longrightarrow}_{N\to +\infty} 0$ strongly in $L^2(\mathbb{S})^d$, which implies that
$$
\int_{\mathbb{S}} w^{\A^N,A^\star}_{\sigma^{ij}} \cdot \left(\left(A^\star\sigma^{ij}\right) \, n\right) \mathop{\longrightarrow}_{N\to +\infty} 0.
$$
We thus deduce that $\dps \lim_{N \to \infty} \mathcal{E}_{\sigma^{ij}}^{\A^N}(A^\star) = \sigma^{ij} \cdot A^\star \sigma^{ij}$.

Similar arguments yield that, for all $1\leq i \leq j \leq d$,
\begin{align*}
  |B| \, \mathcal{E}_{\sigma^{ij}}^{\A^N}(A_1^N)
  &=
  \int_B \sigma^{ij} \cdot \A^N \left(\sigma^{ij}+e\left(w^{\A^N,A_1^N}_{\sigma^{ij}}\right) \right) - \int_{\mathbb{S}} w^{\A^N,A_1^N}_{\sigma^{ij}}\cdot \left(\left(A_1^N\sigma^{ij}\right) \, n \right)
  \\
  &\mathop{\longrightarrow}_{N\to +\infty}
  \int_B \sigma^{ij} \cdot A^\star \left(\sigma^{ij}+e\left(w^{A^\star,A_1^\infty}_{\sigma^{ij}}\right) \right) - \int_{\mathbb{S}} w^{A^\star,A_1^\infty}_{\sigma^{ij}} \cdot \left(\left(A_1^\infty\sigma^{ij}\right) \, n \right)
  \\
  &=
  |B| \, \mathcal{E}_{\sigma^{ij}}^{A^\star}(A_1^\infty)
  \\
  &=
  \int_B \sigma^{ij} \cdot A^\star \sigma^{ij} - \int_B e\left(w^{A^\star,A_1^\infty}_{\sigma^{ij}}\right) \cdot A^\star e\left(w^{A^\star,A_1^\infty}_{\sigma^{ij}}\right) - \int_{\R^d \setminus B} e\left(w^{A^\star,A_1^\infty}_{\sigma^{ij}}\right) \cdot A_1^\infty e\left(w^{A^\star,A_1^\infty}_{\sigma^{ij}}\right),
\end{align*}
where we have used~\eqref{elseE_1} in the last line.

Passing to the limit $N \to \infty$ in~\eqref{eq:covid3}, we thus deduce that
$$
-\sum_{1\leq i\leq j\leq d} \left( \int_B e\left(w^{A^\star,A_1^\infty}_{\sigma^{ij}}\right) \cdot A^\star e\left(w^{A^\star,A_1^\infty}_{\sigma^{ij}}\right) + \int_{\R^d \setminus B} e\left(w^{A^\star,A_1^\infty}_{\sigma^{ij}}\right) \cdot A_1^\infty e\left(w^{A^\star,A_1^\infty}_{\sigma^{ij}}\right)\right) \geq 0.
$$
Since $A^\star$ and $A_1^\infty$ belong to $\mathcal{M}$ and are thus bounded away from 0, we obtain that $e\left(w^{A^\star,A_1^\infty}_{\sigma^{ij}} \right) = 0$ on $\R^d$ for all $1 \leq i,j \leq d$. Using Lemma~\ref{lem:use} and the fact that $w_{\sigma^{ij}}^{A^\star,A_1^\infty} \in V_0$, this implies that $w_{\sigma^{ij}}^{A^\star,A_1^\infty} = 0$. Since $\left( \sigma^{ij}\right)_{1\leq i \leq j \leq d}$ forms a basis of $\cS$ and since the mapping $\cS \ni \sigma \mapsto w_\sigma^{A^\star,A_1^\infty}$ is linear, we obtain that $w_\sigma^{A^\star,A_1^\infty} = 0$ for all $\sigma \in \cS$. We conclude that $A_1^\infty = A^\star$ using Lemma~\ref{same-tensor}.

\medskip

\noindent
{\bf Approximation 2:} In view of the definition~\eqref{tensor2} of $A_2^N$ and using Lemma~\ref{converge-in-out}, we observe that, for all $\sigma \in \cS$,
$$
A_2^N \sigma = \frac{1}{|B|} \int_B \A^N \left(\sigma + e\left(w_\sigma^{\A^N,A_1^N} \right)\right) \mathop{\longrightarrow}_{N\to +\infty} \frac{1}{|B|} \int_B A^\star \left(\sigma + e\left(w_\sigma^{A^\star,A_1^\infty} \right)\right).
$$
Using next that $A_1^\infty = A^\star$ and hence that $w_\sigma^{A^\star,A_1^\infty} = 0$, we deduce that $\dps A_2^N \sigma$ converges to $A^\star\sigma$ for all $\sigma \in \cS$, and thus that $\dps A_2^N \mathop{\longrightarrow}_{N\to +\infty} A^\star$.

\medskip

\noindent
{\bf Approximation 3:} For all $1\leq i\leq j \leq d$, using~\eqref{elseE_2} and the same arguments as above, we have
\begin{align*}
  \sigma^{ij} \cdot A_3^N \sigma^{ij}
  &=
  \mathcal{E}_{\sigma^{ij}}^{\A^N}(A_1^N)
  \\
  & =
  \frac{1}{|B|} \left[ \int_B \sigma^{ij} \cdot \A^N \left(\sigma^{ij}+e\left(w^{\A^N,A_1^N}_{\sigma^{ij}}\right) \right) - \int_{\mathbb{S}} w^{\A^N,A_1^N}_{\sigma^{ij}} \cdot \left(\left(A_1^N\sigma^{ij}\right) \, n \right) \right]
  \\
  & \mathop{\longrightarrow}_{N\to +\infty}
  \frac{1}{|B|} \int_B \sigma^{ij} \cdot A^\star \sigma^{ij} = \sigma^{ij}\cdot A^\star \sigma^{ij}.
\end{align*}
We thus infer that $\dps A_3^N \mathop{\longrightarrow}_{N\to +\infty} A^\star$.

\medskip

\noindent
{\bf Approximation 4:} Since $A_4^N \in \mathcal{M}$ for all $N$, there exists $A^\infty_4 \in \cM$ such that (up to the extraction of a subsequence) $\dps A_4^N \mathop{\longrightarrow}_{N\to +\infty} A_4^\infty$. In view of~\eqref{tensor4} and~\eqref{elseE_2}, we have, for any $\sigma \in \cS$,
$$
\sigma \cdot A_4^N \sigma = \mathcal{E}_\sigma^{\A^N}(A_4^N) = \frac{1}{|B|} \int_B \sigma \cdot \A^N \left(\sigma+e\left(w_\sigma^{\A^N,A_4^N}\right)\right) - \frac{1}{|B|} \int_{\mathbb{S}} w_\sigma^{\A^N,A_4^N} \cdot \left(\left(A_4^N \sigma\right) \, n\right).
$$
Taking the limit $N \to +\infty$ and using Lemma~\ref{converge-in-out} yields that, for all $\sigma \in \cS$,
\begin{equation} \label{eq:covid4}
\sigma \cdot A_4^\infty \sigma = \frac{1}{|B|} \int_B \sigma \cdot A^\star \left(\sigma+e\left(w_\sigma^{A^\star,A_4^\infty}\right)\right) - \frac{1}{|B|} \int_{\mathbb{S}} w_\sigma^{A^\star,A_4^\infty} \cdot \left(\left(A_4^\infty \sigma\right) \, n \right) = \mathcal{E}_\sigma^{A^\star}(A_4^\infty),
\end{equation}
where we have used~\eqref{elseE_2} for the last equality and where $w_\sigma^{A^\star,A_4^\infty}$ is the unique solution in $V_0$ to 
$$
-{\rm div} \left[ \cA^{A^\star,A_4^\infty}\left(\sigma+e\left(w_\sigma^{A^\star,A_4^\infty}\right) \right) \right] = 0 \quad \text{in $\mathcal{D}'(\R^d)^d$}. 
$$
Using~\eqref{elseE_1}, we deduce from~\eqref{eq:covid4} that
$$
  \frac{1}{|B|} \int_B \sigma \cdot (A^\star - A_4^\infty) \sigma
  =
  \frac{1}{|B|} \int_B e\left(w_\sigma^{A^\star,A_4^\infty} \right) \cdot A^\star e\left(w_\sigma^{A^\star,A_4^\infty} \right) + \frac{1}{|B|} \int_{\R^d \setminus B} e\left(w_\sigma^{A^\star,A_4^\infty} \right) \cdot A_4^\infty e\left(w_\sigma^{A^\star,A_4^\infty} \right) \geq 0,
$$
which implies that $\sigma \cdot A^\star \sigma \geq \sigma \cdot A_4^\infty \sigma$ for all $\sigma \in \cS$. 

\medskip

We next infer from~\eqref{eq:covid4} that
\begin{equation}\label{smaller0}
\sigma \cdot A_4^\infty \sigma = \mathcal{E}_\sigma^{A^\star,A_4^\infty}\left(w_\sigma^{A^\star,A_4^\infty}\right) \geq \mathcal{E}_\sigma^{A_4^\infty,A_4^\infty}\left(w_\sigma^{A^\star,A_4^\infty}\right) \geq \mathcal{E}_\sigma^{A_4^\infty,A_4^\infty}\left(w_\sigma^{A^\infty_4,A_4^\infty}\right),
\end{equation}
where we have successively used that $A^\star \geq A_4^\infty$ and the fact that $w_\sigma^{A_4^\infty,A_4^\infty}$ is the minimizer of $\mathcal{E}_\sigma^{A_4^\infty,A_4^\infty}$ in $V_0$. Using Lemma~\ref{same-tensor}, we of course have $w_\sigma^{A^\infty_4,A_4^\infty} = 0$, and hence $\mathcal{E}_\sigma^{A_4^\infty,A_4^\infty}\left(w_\sigma^{A^\infty_4,A_4^\infty}\right) = \sigma \cdot A_4^\infty \sigma$. We thus obtain that, in~\eqref{smaller0}, all inequalities are actually equalities, and hence that $\dps \mathcal{E}_\sigma^{A_4^\infty,A_4^\infty}\left(w_\sigma^{A^\star,A_4^\infty}\right) = \sigma \cdot A_4^\infty \sigma$. Since $\sigma \cdot A_4^\infty \sigma$ is the minimum of $\dps \mathcal{E}_\sigma^{A_4^\infty,A_4^\infty}$ in $V_0$, which is uniquely attained at $w_\sigma^{A^\infty_4,A_4^\infty}$, this implies that $w_\sigma^{A^\star,A_4^\infty} = w_\sigma^{A^\infty_4,A_4^\infty} = 0$. Using again Lemma~\ref{same-tensor}, we deduce from the fact that $w_\sigma^{A^\star,A_4^\infty} = 0$ for all $\sigma \in \cS$ that $A_4^\infty = A^\star$. This concludes the proof of Proposition~\ref{prop:convergence}.
\end{proof}


\section{Case of isotropic materials} \label{sec:iso}

In the general case, we are not able to prove that there exists some tensor $A_4^N$ satisfying the self-consistent formula~\eqref{tensor4}. In the case of a scalar diffusion equation, we had the same difficulty (see~\cite[Section~3.6]{PART1}). However, in the case when the homogenized material is isotropic, we were able to show a weaker existence result (by considering a formula which is self-consistent in a weaker sense) and to prove that this self-consistent approximation converges to the homogenized coefficient in the limit $N \to \infty$ (see~\cite[Proposition~3.7]{PART1}). Our aim in Section~\ref{sec:pro} is to follow the same path, that is to consider isotropic elastic materials and to replace the self-consistent formula~\eqref{tensor4} by~\eqref{eq:mu} below. To do so, we need some preliminary results on isotropic materials, which are collected in this section and which will be useful to prove (using arguments specific to the elasticity case) our results in Section~\ref{sec:pro}. From now on, we assume that $d = 3$. Our main result in this section is Lemma~\ref{Elsheby}. It implies the relations~\eqref{energy-isotropic_1} and~\eqref{energy-isotropic_3}, which are the only results of this Section~\ref{sec:iso} that we use in Section~\ref{sec:pro}.

\subsection {Homogenization in isotropic elasticity}

We consider the specific situation where the material is isotropic. A symmetric fourth-order tensor $A \in \widetilde{\cM}$ is said to be an isotropic tensor if there exists real parameters $\lambda$ and $\mu$ (called the Lam\'e coefficients of $A$) such that 
\begin{equation}\label{restriction-lame}
\mu>0, \qquad 2\mu+3\lambda>0,
\end{equation}
and
$$
\forall \sigma \in \cS, \quad A \sigma = 2 \mu \, \sigma + \lambda \, {\rm Tr} (\sigma) \, {\rm Id},
$$
where ${\rm Id}$ is the $3 \times 3$ identity matrix. In the following, for any $\lambda,\mu \in \R$ satisfying~\eqref{restriction-lame}, we denote by $A^{\lambda,\mu}_{\rm iso} \in \widetilde{\mathcal{M}}$ the unique fourth-order tensor such that 
\begin{equation}\label{eq:def_A-is-lambda-mu}
\forall \sigma \in \cS, \quad A^{\lambda,\mu}_{\rm iso} \sigma = 2 \mu \, \sigma + \lambda \, {\rm Tr} (\sigma) \, {\rm Id}.
\end{equation}
The set of isotropic tensors is denoted by $\widetilde{\mathcal{I}}\subset \widetilde{\cM}$. We recall that a tensor $A \in \widetilde{\cM}$ is isotropic (i.e. belongs to $\widetilde{\mathcal{I}}$) if and only if, for any orthogonal matrix $U \in \R^{3\times 3}$, it holds that
$$
\forall \sigma \in \mathcal{S}, \quad (U^{-1} \sigma U) \cdot A (U^{-1} \sigma U) = \sigma \cdot A \sigma. 
$$
We denote by $\mathcal{I} := \mathcal{M} \cap \widetilde{\mathcal{I}}$, where we recall that $\mathcal{M}$ is the set of symmetric fourth-order tensors satisfying the ellipticity and boundedness assumptions~\eqref{eq:ellipticity}. It can be easily checked that 
\begin{equation} \label{eq:borne_I}
\mathcal{I} = \left\{ A^{\lambda,\mu}_{\rm iso}, \quad \alpha \leq 2\mu \leq \beta, \quad \alpha \leq 2\mu+ 3\lambda \leq \beta \right\}. 
\end{equation}
%
%
We recall the following definition of a {\em stochastically isotropic} stationary random field~\cite[Section~12.3]{JKO}.

\begin{definition}\label{def:isorand}
A stationary, scalar-valued random field $a \in L^1_{\rm loc}(\R^d,L^1(\Omega))$ is called stochastically isotropic if, for any orthogonal matrix $U \in \R^{d\times d}$, there exists a mapping $\theta_U: \Omega \to \Omega$ such that $\theta_U$ preserves the invariant measure on $\Omega$ and such that $a(Uy, \omega) = a(y,\theta_U(\omega))$ a.s. and for almost all $y\in \R^d$.
\end{definition}

Let us recall the following theorem~\cite[Theorem~12.5]{JKO}:

\begin{theorem}
Let $\mu$ and $\lambda$ in $L^\infty(\R^3,L^1(\Omega))$ be two stationary and stochastically isotropic random fields in the sense of Definitions~\ref{De:stationary} and~\ref{def:isorand}. We furthermore assume that $\mu$ and $\lambda$ are bounded from above and below such that $A_{\rm iso}^{\lambda(x,\omega),\mu(x,\omega)} \in \mathcal{I}$ for almost all $x\in \R^3$ and almost surely.
Let $\A \in L^\infty\big(\R^3, L^1\left(\Omega,\mathcal{I}\right)\big)$ be defined by $\A(x,\omega) = A_{\rm iso}^{\lambda(x,\omega),\mu(x,\omega)}$. Stated otherwise, we define $\A(x,\omega)$ by
$$
\forall \sigma \in \mathcal{S}, \qquad \A(x,\omega)\sigma := 2 \mu(x,\omega) \, \sigma + \lambda(x,\omega) \, {\rm Tr} (\sigma) \, {\rm Id}.
$$
Consider the rescaled field defined by $\A^\eps(x,\omega) = \A(x/\eps,\omega)$ for any $\eps > 0$, which satisfies $\A^\eps(x,\omega) \in \cM$ for any $\eps>0$, almost surely and almost everywhere.

Then, the family $\left( \A^\eps(\cdot,\omega) \right)_{\eps >0}$ G-converges (in the sense of Definition~\ref{Def-G}) to some homogenized limit $A^\star$ which is deterministic, constant, belongs to $\cM$ and is isotropic (and thus belongs to $\mathcal{I}$). 
\end{theorem}




\subsection{Properties of layer potentials and vector spherical harmonics in $\R^3$} \label{sec:layers}

Vector spherical harmonics are important tools in the context of linear isotropic elasticity in dimension $d=3$ and we use them extensively in the remainder of this Section~\ref{sec:iso}. We collect here some properties useful to prove our subsequent results. In what follows, we make use of spherical coordinates, which are classically defined by the following diffeomorphism:
$$
\Phi: \left\{
\begin{array}{ccc}
  \R_+^\star \times (0,\pi) \times (0,2\pi) & \to & \mathcal{O} := \R^3 \setminus \{(x_1,0,x_3), \; x_1\in \R_+, \; x_3 \in \R\}
  \\
  (r,\theta,\phi) & \mapsto & \left(r \sin \theta \cos \phi, r \sin\theta \sin \phi, r\cos \theta \right)
\end{array}
\right. .
$$
For any $x \in \mathcal{O}$, we denote by $e_r(x)$, $e_\theta(x)$ and $e_\phi(x)$ the associated radial, polar and azimuthal unit vectors. For any smooth function $Z : \R^3 \to \R$, we denote by $\overline{Z}$ the function defined by
$$
\forall (\theta,\phi) \in (0,\pi) \times (0,2\pi), \quad \overline{Z}(\theta,\phi) = Z(\Phi(1,\theta,\phi)).
$$
In addition, we define the spherical gradient of $Z$ at a point $x \in \mathbb{S} \cap \mathcal{O}$ by
$$
\nabla_s Z(x) = e_\theta(x) \, \partial_\theta \overline{Z}(\theta,\phi) + e_\phi(x) \frac{1}{\sin \theta} \, \partial_\phi \overline{Z}(\theta,\phi), \quad \mbox{where} \quad (1,\theta,\phi) = \Phi^{-1}(x).
$$
We note that $\dps \nabla Z = \nabla_s Z + \big( e_r \cdot \nabla Z \big) \, e_r$ and that $e_r \cdot \nabla_s Z = 0$. If the function $Z$ is sufficiently smooth, its spherical gradient can be extended by continuity over the entire sphere $\mathbb{S}$.  

\medskip

The (standard, scalar-valued) spherical harmonics on the unit sphere $\mathbb{S}$ are denoted by $(Y_{lm})_{l\geq 0}^{|m|\leq l}$. They belong to $\mathcal{C}^\infty(\mathbb{S})$ and are normalized in the sense that 
$$
\langle Y_{lm}, Y_{l'm'} \rangle_{\mathbb{S}} = \int_{\mathbb{S}} Y_{lm} \, Y_{l'm'} = \delta_{ll'} \, \delta_{mm'}.
$$
%
%
For all degree $l\in \N$ and all order $m\in \Z$ such that $|m|\leq l$, we introduce the following functions $V_{lm}$, $W_{lm}$ and $X_{lm}$, defined on $\mathbb{S} \cap \mathcal{O}$ with values in $\R^3$, and which are called the vector spherical harmonics:
\begin{align*} 
  V_{lm} &= \nabla_s Y_{lm} - (l+1) \, Y_{lm} \, e_r,
  \\
  W_{lm} &= \nabla_s Y_{lm} + l \, Y_{lm} \, e_r,
  \\
  X_{lm} &= e_r \times \nabla_s Y_{lm},
\end{align*}
where we recall that $e_r$ is the unit radial vector and where the symbol $\times$ represents the cross product in $\R^3$. By continuity, these functions can be extended over the entire sphere $\mathbb{S}$.
We refer to~\cite{Hill} for more details about vector spherical harmonics. 

\medskip

We now give a brief introduction to the layer potentials and the associated boundary operators in isotropic elasticity, and refer to~\cite{SX} for comprehensive details. Let $\lambda,\mu \in \R$ such that $\mu>0$ and $2\mu + 3\lambda>0$.

We first introduce the single layer potential $\widetilde{\mathcal{V}}^{\lambda,\mu} : H^{-1/2}(\mathbb{S})^3 \to H^1_{\rm loc}(\R^3 \setminus \mathbb{S})^3$ defined by
\begin{equation}\label{single-layer}
  \forall \phi \in H^{-1/2}(\mathbb{S})^3, \qquad \forall x \in \R^3 \setminus \mathbb{S}, \qquad \left( \widetilde{\mathcal{V}}^{\lambda,\mu} \phi \right)(x) = \int_{\mathbb{S}} G^{\lambda,\mu}(x,y) \, \phi(y) \, dy,
\end{equation}
where $G^{\lambda,\mu} = (G^{\lambda,\mu}_{ij})_{1\leq i, j\leq 3}$ is the matrix-valued Green function defined as follows: for all $x = (x_i)_{1\leq i \leq 3} \neq y = (y_i)_{1\leq i \leq 3} \in \R^3$ and all $1\leq i,j \leq 3$,
$$
G^{\lambda,\mu}_{ij}(x,y) := \frac{1}{8 \pi \mu |x-y|} \left( \frac{\lambda+3\mu}{\lambda+2\mu} \, \delta_{ij} + \frac{\lambda+\mu}{\lambda+2\mu} \, \frac{(x_i-y_i)(x_j-y_j)}{|x-y|^2} \right). 
$$
This function is the Green function of the elasticity operator $L^{\lambda,\mu}$ defined by $\dps L^{\lambda,\mu} u = - {\rm div} \Big( 2 \mu \, e(u) + \lambda \, {\rm Tr}(e(u)) \, {\rm Id} \Big)$ for any sufficiently regular $u$, in the sense that, for any $1 \leq i \leq 3$, we have
$$
L^{\lambda,\mu} \, G^{\lambda,\mu}_i(\cdot,y) = \delta_y \, e_i \qquad \mbox{in $\R^3$},
$$
where $\delta_y$ in the Dirac mass in $y$ and where $G_i^{\lambda,\mu}(x,y) := \left( G_{ji}^{\lambda,\mu}(x,y) \right)_{1\leq j \leq 3}$ is the $i^{\rm th}$ column of the $3 \times 3$ symmetric matrix $G^{\lambda,\mu}(x,y)$.

\medskip

For all $\phi \in H^{-1/2}(\mathbb{S})$, the function $\widetilde{\mathcal{V}}^{\lambda,\mu} \phi$ is continuous across the sphere $\mathbb{S}$. We can therefore define the single layer boundary operator $\mathcal{V}^{\lambda,\mu} : H^{-1/2}(\mathbb{S})^3 \to H^{1/2}(\mathbb{S})^3$ as follows:
\begin{equation}\label{single-boundary}
 \forall \phi\in H^{-1/2}(\mathbb{S})^3, \qquad \forall x \in \mathbb{S}, \qquad \left( \mathcal{V}^{\lambda,\mu} \phi \right)(x) = \int_{\mathbb{S}} G^{\lambda,\mu}(x,y) \, \phi(y) \, dy.
\end{equation}
Moreover, we have the following property:
\begin{equation} \label{eq:covid18} 
  \forall \phi \in H^{-1/2}(\mathbb{S})^3, \qquad \widetilde{\mathcal{V}}^{\lambda,\mu} \phi \in L^2_{\rm loc}(\R^3)^3 \qquad \text{and} \qquad \nabla \left( \widetilde{\mathcal{V}}^{\lambda,\mu} \phi \right) \in L^2(\R^3)^{3 \times 3}.
\end{equation}
  
\medskip

The vector spherical harmonics are eigenvectors of the single layer boundary operator (see~\cite[Theorem~5.1]{SX}): we have
\begin{equation} \label{eigenS}
  \begin{array}{rcl}
  \mathcal{V}^{\lambda,\mu} \, V_{lm} &=& \dps \frac{(3l+1)\mu + l\lambda}{(2l+3)(2l+1) \mu(2\mu+\lambda)} \, V_{lm},
  \\ \noalign{\vskip 3pt}
  \mathcal{V}^{\lambda,\mu} \, W_{lm} &=& \dps \frac{(3l+2)\mu + (l+1)\lambda}{(2l-1)(2l+1) \mu(2\mu+\lambda)} \, W_{lm},
  \\ \noalign{\vskip 3pt}
  \mathcal{V}^{\lambda,\mu} \, X_{lm} &=& \dps \frac{1}{\mu (2l+1)} \, X_{lm},
  \end{array}
\end{equation}
for all $l\in\N$ and $m\in \Z$ such that $|m|\leq l$.

\medskip

For all $u \in H^1_{\rm loc}(\R^3)^3$ such that ${\rm div} \big(2\mu \, e(u) + \lambda \, {\rm Tr}(e(u)) \, {\rm Id} \big) \in L^2_{\rm loc}(\R^3)^3$, we denote by $\mathcal{T}_-^{\lambda,\mu}u$ (respectively $\mathcal{T}_+^{\lambda,\mu} u$) the interior (respectively exterior) trace of the function
$$
\big(2\mu \, e(u)+ \lambda \, {\rm Tr}(e(u)) \, {\rm Id}\big) n
$$
on the sphere $\mathbb{S}$ with respect to its outward pointing normal vector $n(x) = x$ at point $x \in \mathbb{S}$.
We have the following relation (see~\cite{BUI} or~\cite[Theorem~3.1]{SX}): 
\begin{equation}\label{ntod}
 \forall \phi \in H^{-1/2}(\mathbb{S})^3, \qquad \mathcal{T}_\mp^{\lambda,\mu} \, \widetilde{\mathcal{V}}^{\lambda,\mu} \, \phi = \pm \frac{1}{2} \phi + \mathcal{D}^{\star,\lambda,\mu} \, \phi, 
\end{equation}
where the (vector-valued) operator $\mathcal{D}^{\star,\lambda,\mu}$ is defined on $H^{-1/2}(\mathbb{S})^3$ by
\begin{equation}\label{eq:double}
 \forall \phi \in H^{-1/2}(\mathbb{S})^3, \quad \forall 1\leq i \leq 3, \quad \forall x \in \mathbb{S}, \qquad \left( \mathcal{D}^{\star,\lambda,\mu} \, \phi\right)_i(x) = \int_{\mathbb{S}} \left[ \mathcal{T}^{\lambda,\mu} \left( G_i^{\lambda,\mu}(\cdot,y) \right) \right](x) \cdot \phi(y) \, dy,
\end{equation}
where we recall that $G_i^{\lambda,\mu}$ is the $i^{\rm th}$ column of the $3 \times 3$ matrix $G^{\lambda,\mu}$ and where the above integral is to be understood in the sense of principal value (see also~\cite[Equation~(14)]{BUI} and~\cite[Equations~(3.4) and~(3.5)]{SX}). 

The operator $\mathcal{D}^{\star,\lambda,\mu}$ is the adjoint operator of the so-called double layer boundary operator $\mathcal{D}^{\lambda,\mu}$ which is discussed in details in~\cite{SX}. The vector spherical harmonics are also eigenvectors of $\mathcal{D}^{\star,\lambda,\mu}$ (see~\cite[Corollary~5.1]{SX}):
\begin{equation} \label{eigenD}
  \begin{array}{rcl}
  \mathcal{D}^{\star,\lambda,\mu} \, V_{lm} &=& \dps - \frac{2(2l^2+6l+1)\mu - 3\lambda}{2(2l+1)(2l+3) (2\mu+\lambda)} \, V_{lm},
  \\ \noalign{\vskip 3pt}
  \mathcal{D}^{\star,\lambda,\mu} \, W_{lm} &=& \dps \frac{2(2l^2-2l-3)\mu - 3\lambda}{2(2l+1)(2l-1) (2\mu+\lambda)} \, W_{lm},
  \\ \noalign{\vskip 3pt}
  \mathcal{D}^{\star,\lambda,\mu} \, X_{lm} &=& \dps \frac{1}{2\mu (2l+1)} \, X_{lm}. 
  \end{array}
\end{equation}
Let us now introduce the following auxiliary functions: for all $1\leq i\neq j \leq 3$ and all $x\in \mathbb{S}$, we set
\begin{align*}
  Z_0(x) & := x = \sum_{i=1}^3 x_i \, {\bf e}_i,
  \qquad
  Z_i(x) := x_i \, {\bf e}_i - \frac{1}{3} \, x,
  \\
  Z_{ij}(x) & := x_i \, {\bf e}_j - x_j \, {\bf e}_i,
  \qquad
  R_{ij}(x) := x_i \, {\bf e}_j + x_j \, {\bf e}_i.
\end{align*}
Using~\cite[Appendix~A]{SX}, we have that $Z_0$ is colinear to the vector spherical harmonic $V_{00}$. Likewise, the function $Z_{23}$ (resp. $Z_{12}$, $Z_{13}$) is colinear to $X_{11}$ (resp. $X_{10}$, $X_{1,-1}$). In addition, the function $R_{12}$ (resp. $R_{23}$, $R_{13}$) is colinear to $W_{2,-2}$ (resp. $W_{2,-1}$, $W_{2,1}$). Lastly, the function $Z_3$ is colinear to $W_{20}$. Using again~\cite[Appendix~A]{SX}, we have that, for any $\dps x = \sum_{i=1}^3 x_i \, {\bf e}_i \in \mathbb{S}$,
$$
Z_1(x) = \frac{1}{6} \, (x_1,x_2,-2x_3)^T + \frac{1}{2} \, (x_1,-x_2,0)^T = - \frac{1}{3} \, \sqrt{\frac{\pi}{5}} \ W_{20}(x) + \sqrt{\frac{\pi}{15}} \ W_{22}(x),
$$
which shows that the function $Z_1$ is a linear combination of $W_{20}$ and $W_{22}$. Since $Z_1 + Z_2 + Z_3 = 0$, the same property holds for $Z_2$. In view of~\eqref{eigenS} and~\eqref{eigenD}, we thus have, for all $1\leq i \neq j \leq 3$,
\begin{equation} \label{eq:covid5}
  \begin{array}{rcl}
    R_{ij}, Z_i & \in & \dps {\rm Ker}\left(\mathcal{V}^{\lambda,\mu} - \frac{8\mu+3\lambda}{15\mu(2\mu+\lambda)} \right) \cap {\rm Ker}\left(\mathcal{D}^{\star,\lambda,\mu} - \frac{2\mu-3\lambda}{30(2\mu+\lambda)} \right),
    \\ \noalign{\vskip 3pt}
    Z_{ij} & \in & \dps {\rm Ker}\left(\mathcal{V}^{\lambda,\mu} - \frac{1}{3\mu}\right) \cap {\rm Ker}\left(\mathcal{D}^{\star,\lambda,\mu} - \frac{1}{6\mu} \right),
    \\ \noalign{\vskip 3pt}
    Z_0 & \in & \dps {\rm Ker}\left(\mathcal{V}^{\lambda,\mu} - \frac{1}{3(2\mu+\lambda)} \right) \cap {\rm Ker}\left(\mathcal{D}^{\star,\lambda,\mu} - \frac{-2\mu+3\lambda}{6(2\mu+\lambda)} \right).
  \end{array}
\end{equation}
These properties are collected in Table~\ref{table1}.


\bigskip

\begin{table}[htbp]
\centering
\begin{tabular}{|c|c|c|c|}
  \hline 
  \diagbox{operator}{eigenvector} & $Z_0$ & $Z_{ij}$ ($i \neq j$) & $R_{ij}$ (for $i \neq j$) and $Z_i$
  \\
  \hline \noalign{\vskip 3pt}
  $\mathcal{V}^{\lambda,\mu}$ & $\dps \frac{1}{3(2\mu+\lambda)}$ & $\dps \frac{1}{3\mu}$ & $\dps \frac{8\mu+3\lambda}{15\mu(2\mu+\lambda)}$
  \\
  \noalign{\vskip 3pt} \hline \noalign{\vskip 3pt}
  $\mathcal{D}^{\star,\lambda,\mu}$ & $\dps \frac{-2\mu+3\lambda}{6(2\mu+\lambda)}$ & $\dps \frac{1}{6\mu}$ & $\dps \frac{2\mu-3\lambda}{30(2\mu+\lambda)}$
  \\
  \noalign{\vskip 3pt} \hline
\end{tabular}
\caption{Overview of some eigenvalues of $\mathcal{V}^{\lambda,\mu}$ and $\mathcal{D}^{\star,\lambda,\mu}$, following~\eqref{eq:covid5}. \label{table1}}
\end{table}
 

\bigskip

Let us now prove the following auxiliary lemma (in the sequel, we only use the first part of the lemma, regarding the case of a symmetric matrix $C$; we state here the second part regarding the case of a skew-symmetric matrix $S$ for the sake of completeness).

\begin{lemma}
\label{phi-and-varphi}
Let $\lambda,\mu \in \R$ such that $\mu>0$ and $2\mu + 3 \lambda>0$. Let $C = (C_{ij})_{1\leq i, j\leq 3} \in \R^{3\times 3}$ be a symmetric matrix, and let $p(x) = Cx$ for all $x\in \R^3$. We define the function $\varphi_C^{\lambda,\mu}$ by 
\begin{equation} \label{eq:def_varphiC}
\forall x \in \R^3, \qquad \varphi_C^{\lambda,\mu}(x) = \left(\frac{15\mu(2\mu+\lambda)}{8\mu+3\lambda} \, C + \frac{3(\mu+\lambda) (2\mu+\lambda)}{8\mu+3\lambda} \ {\rm Tr}(C) \, {\rm Id} \right) x.
\end{equation}
It holds that
\begin{equation}\label{pi}
\mathcal{V}^{\lambda,\mu} \, \varphi_C^{\lambda,\mu} = p|_{\mathbb{S}}.
\end{equation}
Let now $S = (S_{ij})_{1\leq i, j\leq 3}\in \R^{3\times 3}$ be a skew-symmetric matrix, and let $q(x) = Sx$ for all $x\in \R^3$. We define the function $\psi^\mu_S$ by
\begin{equation} \label{eq:def_psiC}
\forall x \in \R^3, \qquad \psi^\mu_S(x) = 3 \mu \, S x.
\end{equation}
Then, it holds that
\begin{equation}\label{qi}
\mathcal{V}^{\lambda,\mu} \, \psi^\mu_S = q|_{\mathbb{S}}.
\end{equation}
\end{lemma}

\begin{proof}
We only detail here the computations yielding~\eqref{pi}. Similar but much simpler computations yield~\eqref{qi}. For all $x = (x_i)_{1\leq i \leq 3} \in \R^3$, let us define
$$
\widehat{p}^C(x) = \sum\limits_{i=1}^3 C_{ii} \, x_i \, {\bf e}_i, \qquad \widetilde{p}^C(x) = \sum\limits_{1\leq i\neq j\leq 3} C_{ij} \, x_j \, {\bf e}_i, 
$$
so that $p = \widehat{p}^C + \widetilde{p}^C$. Since $C$ is symmetric, we have $\dps \widetilde{p}^C = \sum_{ 1\leq i<j\leq 3} C_{ij} \, R_{ij}$. In view of~\eqref{eq:covid5}, we hence obtain $\dps \widetilde{p}^C \in {\rm Ker}\left(\mathcal{V}^{\lambda,\mu} - \frac{8\mu+3\lambda}{15\mu(2\mu+\lambda)}\right)$ and therefore
\begin{equation}\label{sum-fori<j}
  \widetilde{p}^C = \frac{15\mu(2\mu+\lambda)}{8\mu+3\lambda} \ \mathcal{V}^{\lambda,\mu} \, \widetilde{p}^C. 
\end{equation}
%
%
Moreover, $\dps \widehat{p}^C = \sum_{i=1}^3 C_{ii} \left( Z_i + \frac{1}{3} Z_0\right) = \frac{1}{3} ({\rm Tr} (C)) \, Z_0 + \widehat{p}^C_Z$ where $\dps \widehat{p}^C_Z := \sum_{i=1}^3 C_{ii} \, Z_i$. In view of~\eqref{eq:covid5}, we observe that
\begin{equation}\label{eq:defphatCZ}
  \widehat{p}^C_Z \in {\rm Ker}\left( \mathcal{V}^{\lambda, \mu} - \frac{8\mu + 3\lambda}{15\mu (2\mu +\lambda)}\right).
\end{equation}
Using that $Z_0$ is also an eigenvector of $\mathcal{V}^{\lambda,\mu}$ (see~\eqref{eq:covid5}), we deduce from~\eqref{eq:defphatCZ} that
\begin{align}
  \widehat{p}^C
  &=
  \frac{1}{3} ({\rm Tr}(C)) \, Z_0 + \widehat{p}^C_Z
  \nonumber
  \\
  &=
  (2\mu+\lambda) \, {\rm Tr}(C) \, \mathcal{V}^{\lambda,\mu}(Z_0) + \frac{15\mu(2\mu+\lambda)}{8\mu+3\lambda} \, \mathcal{V}^{\lambda,\mu}\left(\widehat{p}^C_Z\right)
  \nonumber
  \\
  &=
  \Big( 2\mu + \lambda - \frac{5\mu(2\mu+\lambda)}{8\mu+3\lambda} \Big) \, {\rm Tr}(C) \, \mathcal{V}^{\lambda,\mu}(Z_0) + \frac{15\mu(2\mu+\lambda)}{8\mu+3\lambda} \, \mathcal{V}^{\lambda,\mu} \Big (\frac{1}{3} \, {\rm Tr}(C) \, Z_0 + \widehat{p}^C_Z \Big)
  \nonumber
  \\
  &=
  \frac{3(\mu+\lambda) (2\mu+\lambda)}{8\mu+3\lambda} \, {\rm Tr}(C) \, \mathcal{V}^{\lambda,\mu}(Z_0) + \frac{15\mu(2\mu+\lambda)}{8\mu+3\lambda} \, \mathcal{V}^{\lambda,\mu}\left(\widehat{p}^C\right). \label{sum-fori=j}
\end{align}
Collecting~\eqref{sum-fori<j} and~\eqref{sum-fori=j}, we obtain~\eqref{pi}.
\end{proof}

We now state an identity, which is used in the sequel, and which is a direct consequence of the definition of $\varphi_C^{\lambda,\mu}$ and of the fact that $\dps p = \widetilde{p}^C + \widehat{p}^C_Z + \frac{1}{3} ({\rm Tr}(C)) \, Z_0$ (where $\widetilde{p}^C$ and $\widehat{p}^C_Z$ are defined in the proof of Lemma~\ref{phi-and-varphi}): for any $\lambda,\mu\in\R$ satisfying~\eqref{restriction-lame} and for any symmetric matrix $C\in \R^{3\times 3}$, it holds that
\begin{equation}\label{simple-p}
\varphi_C^{\lambda,\mu} = (2\mu+\lambda) \, ({\rm Tr}(C)) \, Z_0 + \frac{15\mu(2\mu+\lambda)}{8\mu+3\lambda} \, \left( \widetilde{p}^C + \widehat{p}^C_Z \right).
\end{equation}


\subsection{The Eshelby's problem}

For the proofs of some results in Section~\ref{sec:pro}, it will be useful to express the solution of the well-known Eshelby's problem using the single layer potential. This is the aim of this section, which culminates in the relations~\eqref{energy-isotropic_1} and~\eqref{energy-isotropic_3}.

\medskip

Let $(\lambda_0,\mu_0) \in \mathbb{R}^2$ and $(\lambda,\mu) \in \mathbb{R}^2$ be Lam\'e coefficients such that $\mu>0$, $\mu_0>0$, $2\mu + 3 \lambda>0$ and $2\mu_0 + 3 \lambda_0 >0$. We denote by $A_0 := A_{\rm iso}^{\lambda_0,\mu_0}$ and by $A := A_{\rm iso}^{\lambda,\mu}$ the corresponding isotropic elasticity tensors (see~\eqref{eq:def_A-is-lambda-mu}). For a fixed $\sigma\in \mathcal{S}$, let $w_\sigma^{A_0,A}$ be the unique solution in $V_0$ to the problem~\eqref{p-corrector} with $\A = A_0$. Otherwise stated, the function $w_\sigma^{A_0,A}$ is the unique solution in $V_0$ to 
\begin{equation}\label{Eshelby-isotropic}
- {\rm div} \Big[ 2\widetilde{\mu} \, \Big(\sigma+e\big(w_\sigma^{A_0,A}\big)\Big) + \widetilde{\lambda} \, {\rm Tr} \Big(\sigma+e\big(w_\sigma^{A_0,A}\big)\Big) \, {\rm Id} \Big] = 0 \quad \mbox{in $\mathcal{D}'(\R^3)^3$}, 
\end{equation}
where
$$
\widetilde{\mu} = \mu_0 \, \chi_B + \mu \, \big(1-\chi_B\big) \quad \mbox{and} \quad \widetilde{\lambda} = \lambda_0 \, \chi_B + \lambda \, \big(1-\chi_B\big),
$$
where $\chi_B$ is the characteristic function of the unit ball $B$. The Eshelby problem~\eqref{Eshelby-isotropic} is equivalent to the following problem:
\begin{subequations}
\begin{align}
  - {\rm div} \Big[ 2\mu_0 \, e\big(w_\sigma^{A_0,A}\big) + \lambda_0 \, {\rm Tr}\Big(e\big(w_\sigma^{A_0,A}\big)\Big) \, {\rm Id} \Big] &= 0 \quad \text{in $B$}, \label{eshelby:a}
  \\
  - {\rm div} \Big[ 2\mu \, e\big(w_\sigma^{A_0,A}\big) + \lambda \, {\rm Tr}\Big(e\big(w_\sigma^{A_0,A}\big)\Big) \, {\rm Id} \Big] &= 0 \quad \text{in $\R^3 \setminus \overline{B}$}, \label{eshelby:b}
  \\
  \llbracket w_\sigma^{A_0,A} \rrbracket &= 0 \quad \text{on $\mathbb{S}$}, \label{eshelby:c}
  \\
  \left\llbracket \mathcal{T} w_\sigma^{A_0,A} \right\rrbracket &= \big( 2(\mu-\mu_0) \, \sigma+ (\lambda-\lambda_0) \, {\rm Tr} (\sigma) \, {\rm Id} \big) n \quad \text{on $\mathbb{S}$}, \label{eshelby:d}
\end{align}
\end{subequations}
where
\begin{gather*}
  \llbracket w_\sigma^{A_0,A} \rrbracket = \gamma^+ w_\sigma^{A_0,A} - \gamma^- w_\sigma^{A_0,A},
  \\
  \llbracket \mathcal{T} w_\sigma^{A_0,A} \rrbracket = \mathcal{T}_-^{\lambda_0,\mu_0} w_\sigma^{A_0,A} - \mathcal{T}_+^{\lambda,\mu} w_\sigma^{A_0,A},
\end{gather*}
where $\gamma^+: H^1_{\rm loc}(\R^3 \setminus \overline{B})^3 \to L^2(\mathbb{S})^3$ (respectively $\gamma^-: H^1(B)^3 \to L^2(\mathbb{S})^3$) denotes the exterior (respectively interior) trace application onto $\mathbb{S}$, and where $\mathcal{T}_-^{\lambda_0,\mu_0} w_\sigma^{A_0,A}$ (respectively $\mathcal{T}_+^{\lambda,\mu} w_\sigma^{A_0,A}$) denotes the interior (respectively exterior) trace of the normal stress associated to the displacement $w_\sigma^{A_0,A}$ and the Lam\'e coefficients $(\lambda_0,\mu_0)$ (respectively $(\lambda,\mu)$). Recall that $\mathcal{T}_-^{\lambda_0,\mu_0}$ has been defined in Section~\ref{sec:layers}. The equation~\eqref{eshelby:c} thus enforces the continuity of the displacement across $\mathbb{S}$, while the equation~\eqref{eshelby:d} enforces the continuity of the normal stress.

\begin{lemma} \label{Elsheby}
Consider the Eshelby's problem~\eqref{eshelby:a}--\eqref{eshelby:d} with $\sigma = \sigma^{kl}$ for some $1\leq k,l \leq 3$, where $\sigma^{kl} \in \mathcal{S}$ is defined by~\eqref{Sij}. This problem admits a unique solution $w_{\sigma^{kl}}^{A_0,A}$ in $V_0$, and this solution reads
\begin{equation}\label{sol-E_1}
\forall x\in \R^3, \qquad w_{\sigma^{kl}}^{A_0,A}(x) = \begin{cases}
  C^{kl} \, x & \mbox{if $|x|\leq 1$},
  \\
 \widetilde{\mathcal{V}}^{\lambda,\mu} \, \varphi^{\lambda,\mu}_{C^{kl}}(x) & \mbox{if $|x|>1$}, 
\end{cases}
\end{equation}
%
%
%
where $C^{kl} \in \R^{3\times 3}$ is the constant symmetric matrix given by
\begin{equation}\label{Cij}
  C^{kl} = \begin{cases}
    \dps 2 (\mu-\mu_0) \, \left(2\mu_0 + \frac{14\mu+9\lambda}{8\mu+3\lambda} \, \mu \right)^{-1} \sigma^{kl} & \mbox{if $k\neq l$},
    \\ \noalign{\vskip 3pt}
    \dps \left( (\lambda-\lambda_0) - \frac{2(\mu-\mu_0) \, \left(\lambda_0 + \frac{6\mu+\lambda}{8\mu+3\lambda} \, \mu\right)}{2\mu_0 + \frac{14\mu+9\lambda}{8\mu+3\lambda} \, \mu} \right) \Big(2\mu_0+3\lambda_0+4\mu\Big)^{-1} \, {\rm Id} + \frac{2(\mu-\mu_0)}{2\mu_0 + \frac{14\mu+9\lambda}{8\mu+3\lambda} \, \mu} \, \sigma^{kl} & \mbox{if $k=l$}, 
\end{cases}
\end{equation}
and where $\varphi^{\lambda,\mu}_{C^{kl}}$ is defined by~\eqref{eq:def_varphiC}.
\end{lemma}

Although the expression of $C^{kl}$ is tedious, the proof of this lemma, which is given below, is not difficult.



\begin{corollary} \label{Elsheby_id}
Consider the Eshelby's problem~\eqref{eshelby:a}--\eqref{eshelby:d} with $\sigma = {\rm Id}$. The unique solution $w_{\rm Id}^{A_0,A}$ in $V_0$ to that problem reads as
\begin{equation}\label{sol-E_2}
\forall x \in \R^3, \qquad w_{\rm Id}^{A_0,A}(x) = \begin{cases}
  C \, x & \mbox{if $|x|\leq 1$},
  \\
  \widetilde{\mathcal{V}}^{\lambda,\mu} \, \varphi^{\lambda,\mu}_C(x) & \mbox{if $|x|>1$},
\end{cases}
\end{equation}
where the matrix $C$ is given by $\dps C = \frac{2(\mu-\mu_0) + 3(\lambda-\lambda_0)}{2\mu_0 + 3\lambda_0 + 4\mu} \ {\rm Id}$.
\end{corollary}

\medskip



Using~\eqref{elseE_2} and~\eqref{sol-E_1}, we obtain that
\begin{equation} \label{eq:E_general}
\mathcal{E}^{A_0}_{\sigma^{kl}}(A)
=
2\mu_0 \, \sigma^{kl} \cdot (\sigma^{kl} + C^{kl}) + \lambda_0 \, {\rm Tr}(\sigma^{kl} + C^{kl}) \, {\rm Tr}(\sigma^{kl}) - 2\mu \, {\rm Tr}(C^{kl} \, \sigma^{kl}) - \lambda \, {\rm Tr}(\sigma^{kl}) \, {\rm Tr}(C^{kl})
\end{equation}
for any $1 \leq k,l \leq 3$ (and of course a similar expression when $\sigma = {\rm Id}$, using~\eqref{sol-E_2}). Using~\eqref{Cij}, we deduce from~\eqref{eq:E_general} that
\begin{equation}\label{energy-isotropic_1}
  \mathcal{E}^{A_0}_{\sigma^{kl}}(A)
  =
  \mu_0 - \frac{2(\mu-\mu_0)^2}{2\mu_0 + \frac{14\mu+9\lambda}{8\mu+3\lambda} \, \mu} \qquad \mbox{if $k\neq l$},
\end{equation}
and
\begin{multline*} 
  \mathcal{E}^{A_0}_{\sigma^{kk}}(A)
  =
  2\mu_0 + \lambda_0 + \big( 2(\mu_0-\mu) + 3(\lambda_0-\lambda) \big) \ \frac{(\lambda- \lambda_0) - \frac{2(\mu-\mu_0) \, \big(\lambda_0 + \frac{6\mu+\lambda}{8\mu+3\lambda} \, \mu\big)}{2\mu_0 + \frac{14\mu+9\lambda}{8\mu+3\lambda} \, \mu}}{2\mu_0+3\lambda_0+4\mu} \\ + \big( 2(\mu_0-\mu) + (\lambda_0-\lambda) \big) \, \frac{2(\mu-\mu_0)}{2\mu_0 + \frac{14\mu+9\lambda}{8\mu+3\lambda} \, \mu}.
\end{multline*}  
Likewise, using Corollary~\ref{Elsheby_id}, we infer from~\eqref{eq:E_general} the following formula in the case when $\sigma = {\rm Id}$:
\begin{equation}\label{energy-isotropic_3}
  \frac{1}{3} \, \mathcal{E}^{A_0}_{\rm Id}(A)
  =
  2\mu_0 + 3\lambda_0 -\frac{\big( 2(\mu_0-\mu) + 3(\lambda_0-\lambda) \big)^2}{2\mu_0+3\lambda_0+4\mu}.
\end{equation}

\begin{proof}[Proof of Lemma~\ref{Elsheby}]
By construction, the function $w_{\sigma^{kl}}^{A_0,A}$ defined by~\eqref{sol-E_1} belongs to $(H^1(B))^3 \cap (H^1_{\rm loc}(\R^3 \setminus \overline{B}))^3$ and its mean over $B$ vanishes. We now check that it satisfies~\eqref{eshelby:a}--\eqref{eshelby:d}. In the first place, it follows from simple calculations and properties of the single layer potentials that~\eqref{eshelby:a} and~\eqref{eshelby:b} are satisfied. Since the single layer potential is continuous across the sphere $\mathbb{S}$, we have $\gamma^+ w_{\sigma^{kl}}^{A_0,A} = \mathcal{V}^{\lambda,\mu} \varphi^{\lambda,\mu}_{C^{kl}}$. Using Lemma~\ref{phi-and-varphi}, we thus have, for any $x\in \mathbb{S}$, that $\dps \left(\gamma^+ w_{\sigma^{kl}}^{A_0,A}\right)(x) = C^{kl} \, x= \left(\gamma^- w_{\sigma^{kl}}^{A_0,A}\right)(x)$ and~\eqref{eshelby:c} is hence satisfied. We now check that~\eqref{eshelby:d} indeed holds. A direct computation yields that, for all $x\in \mathbb{S}$,
$$
\left( \mathcal{T}_-^{\lambda_0,\mu_0} \, w_{\sigma^{kl}}^{A_0,A} \right)(x) = \left(2\mu_0 \, C^{kl} + \lambda_0 \, {\rm Tr}(C^{kl}) \right) \, x,
$$
where we have used that the normal vector $n$ on $\mathbb{S}$ satisfies $n(x) = x$. To compute the exterior trace of the normal stress, we use~\eqref{ntod}, which reads
$$
\mathcal{T}_+^{\lambda,\mu} \, w_{\sigma^{kl}}^{A_0,A} = -\frac{1}{2} \varphi^{\lambda,\mu}_{C^{kl}} + \mathcal{D}^{\star,\lambda,\mu} \, \varphi^{\lambda,\mu}_{C^{kl}} \qquad \mbox{on $\mathbb{S}$}.
$$
In view of~\eqref{simple-p}, we have
$$
\varphi^{\lambda,\mu}_{C^{kl}} = \frac{15\mu(2\mu+\lambda)}{8\mu+3\lambda} \left(\widetilde{p}^{C^{kl}} + \widehat{p}_Z^{C^{kl}}\right) + (2\mu+\lambda) \, {\rm Tr}(C^{kl}) \, Z_0.  
$$
Using the spectral properties of $\mathcal{D}^{\star,\lambda,\mu}$ (see~\eqref{eq:covid5} and Table~\ref{table1}), we obtain
\begin{align*}
  & \mathcal{T}_+^{\lambda,\mu} \, w_{\sigma^{kl}}^{A_0,A}
  \\
  &=
  -\frac{1}{2} \varphi^{\lambda,\mu}_{C^{kl}} + \mathcal{D}^{\star,\lambda,\mu} \varphi^{\lambda,\mu}_{C^{kl}}
  \\  
  &=
  \left( -\frac{1}{2} + \frac{2\mu-3\lambda}{30(2\mu+\lambda)} \right) \, \frac{15\mu(2\mu+\lambda)}{8\mu+3\lambda} \left( \widetilde{p}^{C^{kl}} + \widehat{p}_Z^{C^{kl}} \right) + \left( -\frac{1}{2} + \frac{-2\mu+3\lambda}{6(2\mu+\lambda)}\right) \, (2\mu+\lambda) \, {\rm Tr}(C^{kl}) \, Z_0
  \\
  &=
  - \frac{14\mu+9\lambda}{8\mu+3\lambda} \, \mu \left (\widetilde{p}^{C^{kl}} + \widehat{p}_Z^{C^{kl}} + \frac{1}{3} {\rm Tr}(C^{kl}) \, Z_0 \right) - \frac{1}{3} \left(4 - \frac{14\mu+9\lambda}{8\mu+3\lambda} \right) \mu \, {\rm Tr}(C^{kl}) \, Z_0
  \\ 
  &=
  - \frac{14\mu+9\lambda}{8\mu+3\lambda} \, \mu \, C^{kl} \, Z_0 - \frac{6\mu+\lambda}{8\mu+3\lambda} \, \mu \, {\rm Tr}(C^{kl}) \, Z_0.
\end{align*}
We therefore have, for any $x \in\mathbb{S}$, that
\begin{equation}\label{eq:jump}
  \llbracket \mathcal{T} w_{\sigma^{kl}}^{A_0,A} \rrbracket(x) = \Big(2\mu_0 + \frac{14\mu+9\lambda}{8\mu+3\lambda} \, \mu \Big) \, C^{kl} \, x + \Big(\lambda_0 + \frac{6\mu+\lambda}{8\mu+3\lambda} \, \mu \Big) \, {\rm Tr}(C^{kl}) \, x.
\end{equation}
Inserting the definition~\eqref{Cij} of $C^{kl}$ in~\eqref{eq:jump}, we obtain~\eqref{eshelby:d} for $\sigma = \sigma^{kl}$.

\medskip

We are now left with showing that $w_{\sigma^{kl}}^{A_0,A}$ defined by~\eqref{sol-E_1} belongs to $V_0$. In view of~\eqref{eq:covid18} and of the properties already shown in this proof, we already know that $w_{\sigma^{kl}}^{A_0,A} \in L^2_{\rm loc}(\R^3)^3$ and that its mean over $B$ vanishes. Using~\eqref{eshelby:c} and again~\eqref{eq:covid18}, we have that $\nabla w_{\sigma^{kl}}^{A_0,A} \in L^2(\R^3)^{3 \times 3}$. We thus indeed have that $w_{\sigma^{kl}}^{A_0,A} \in V_0$, which concludes the proof of Lemma~\ref{Elsheby}.
\end{proof}

\begin{proof}[Proof of Corollary~\ref{Elsheby_id}]
We note that $\dps \sum_{k=1}^3 \sigma^{kk} = {\rm Id}$. Using the linearity of the Eshelby problem, we thus have $\dps w_{\rm Id}^{A_0,A} = \sum_{k=1}^3 w_{\sigma^{kk}}^{A_0,A}$. We hence obtain that the solution is given by~\eqref{sol-E_2} with $\dps C = \sum_{k=1}^3 C^{kk}$, with $C^{kk}$ given by~\eqref{Cij}. An easy computation yields the claimed expression of the matrix $C$.
\end{proof}





\section{Properties of the embedded corrector problem in isotropic elasticity} \label{sec:pro}

Using the technical tools introduced in Section~\ref{sec:iso}, we now return to the study of embedded corrector problems, focusing on the case of isotropic elasticity in dimension $d=3$. In some particular cases, we are in position to strengthen our previous concavity result (see Lemmas~\ref{lem:strict-concave-lambda} and~\ref{lem:strict-concave-mu} below). Under some additional assumptions, we are then able to show the existence of a self-consistent tensor (in a weaker sense than~\eqref{tensor4}) which again converges to the homogenized tensor (see Proposition~\ref{self-consistent} below). We stress here the fact that the adaptation of~\cite[Proposition~3.7]{PART1} in the thermal case to the current elasticity case requires specific and novel arguments.



\subsection{Strict concavity} \label{sec:strict}

Lemma~\ref{concave} states that, for any $\sigma \in \cS$ and any $\A\in L^\infty(B,\mathcal{M})$, the mapping $\mathcal{M} \ni A \rightarrow \mathcal{E}_\sigma^{\A}(\mathcal{A})$ is concave. In this section, in the particular case of isotropic elasticity in dimension $d=3$, we investigate some cases where this mapping can be proved to be {\em strictly} concave, or can be proved to not be strictly concave.

\medskip

In what follows, instead of working with the Lam\'e coefficients $\mu$ and $\lambda$, we are going to work with $\mu$ and $\kappa = 2\mu+3\lambda$, under the restriction (see~\eqref{restriction-lame}) that $\mu,\kappa > 0$. For any such $\mu$ and $\kappa$, we set
\begin{equation} \label{eq:changement_lame}
  A_{{\rm iso},2}^{\kappa,\mu} := A_{\rm iso}^{\lambda,\mu} \quad \text{with} \quad \lambda = \frac{\kappa - 2\mu}{3},
\end{equation}
where we recall that $A_{\rm iso}^{\lambda,\mu}$ is defined by~\eqref{eq:def_A-is-lambda-mu}. In view of~\eqref{eq:borne_I}, we have 
\begin{equation} \label{eq:borne_I_bis}
  \mathcal{I} = \left\{ A_{{\rm iso},2}^{\kappa,\mu}, \quad \alpha \leq \kappa \leq \beta, \quad \alpha \leq 2\mu \leq \beta \right\}.
\end{equation}
For $A = A_{{\rm iso},2}^{\kappa,\mu}$ and $A_0 = A_{{\rm iso},2}^{\kappa_0,\mu_0}$, the equation~\eqref{energy-isotropic_3} equivalently reads
\begin{equation}\label{energy-isotropic_3_kappa_mu}
  \frac{1}{3} \, \mathcal{E}^{A_0}_{\rm Id}(A)
  =
  \kappa_0 -\frac{\big( \kappa_0-\kappa \big)^2}{\kappa_0+4\mu}.
\end{equation}

\medskip

We begin by a negative result. 

\begin{lemma}
\label{lem:non-strict-concave}
Let $d=3$. Consider the set $\mathcal{I}$ defined by~\eqref{eq:borne_I_bis} and assume that $\beta \geq 5 \alpha/2$. Let $\kappa_0 = \alpha$ and let $\mu_0 \in (\alpha/2,\beta/2)$. Let $A_0 := A_{{\rm iso},2}^{\kappa_0,\mu_0} \in \mathcal{I}$ be the corresponding isotropic elasticity tensor. The mapping $\mathcal{I} \ni A \to \mathcal{E}^{A_0}_{\rm Id}(A)$ defined by~\eqref{eq:Energy} (for the specific case $\sigma = {\rm Id}$ and $\A = A_0$) is {\em not strictly concave}.
\end{lemma}

This result is in sharp contrast with the thermal case, where we are in position to use a {\em strict} concavity argument in the proof of~\cite[Lemma~3.3]{PART1}.

\begin{proof}
  %
  We have set $\kappa_0 = \alpha$ and we further set $\kappa = \beta$, $\overline{\kappa} = (\alpha+\beta)/2$ and $\kappa(t) = t \, \kappa + (1-t) \, \overline{\kappa}$. We do not make precise $\mu_0 \in (\alpha/2,\beta/2)$ (since $\mathcal{E}^{A_0}_{\rm Id}(A)$ does not depend on it) and choose $\mu = 5\alpha/4$ (we have assumed $5\alpha/4 \leq \beta/2$), $\overline{\mu} = \alpha/2$ and $\mu(t) = t \, \mu + (1-t) \, \overline{\mu}$. We consider the function $\dps f(t) = \frac{1}{3} \, \mathcal{E}^{A_0}_{\rm Id}\left(A_{{\rm iso},2}^{\kappa(t),\mu(t)}\right)$ for $t \in [0,1]$, and we underline that $A_0$ and $A_{{\rm iso},2}^{\kappa(t),\mu(t)}$ belong to $\mathcal{I}$ for any $t \in [0,1]$. Furthermore, $A_{{\rm iso},2}^{\kappa(t),\mu(t)} = t \, A_{{\rm iso},2}^{\kappa,\mu} + (1-t) \, A_{{\rm iso},2}^{\overline{\kappa},\overline{\mu}}$ is a convex linear combination of $A_{{\rm iso},2}^{\kappa,\mu} \in \mathcal{I}$ and $A_{{\rm iso},2}^{\overline{\kappa},\overline{\mu}} \in \mathcal{I}$. The macroscopic loading in the embedded corrector problem is $\sigma = {\rm Id}$, which is an isotropic matrix.

  We now show that $f$ is actually an affine function. Using~\eqref{energy-isotropic_3_kappa_mu}, we have $\dps f(t) = \kappa_0 -\frac{\big( \kappa_0-\kappa(t) \big)^2}{\kappa_0+4\mu(t)}$ with
  $$
  \kappa_0-\kappa(t) = \alpha - t \, \beta - (1-t) \frac{\alpha+\beta}{2} = \frac{\alpha-\beta}{2} \, (1+t) 
  $$
  and
  $$
  \kappa_0+4\mu(t) = \alpha + 5 \, t \, \alpha + 2 \, (1-t) \, \alpha = 3 \, \alpha \, (1+t),
  $$
  and thus
  $$
  f(t) = \alpha - \frac{(\alpha-\beta)^2}{12 \, \alpha} \, (1+t). 
  $$
  The function $f$ is thus indeed affine. 
\end{proof}

In Lemma~\ref{lem:non-strict-concave}, we have considered the function of two variables $(\kappa,\mu) \mapsto \mathcal{E}^{A_0}_{\rm Id}(A_{{\rm iso},2}^{\kappa,\mu})$. If we now restrict our attention to the variable $\kappa$, the corresponding map turns out to be strictly concave.

\begin{lemma}
\label{lem:strict-concave-lambda}
Let $d=3$. Let $\kappa_0>0$, $\mu_0>0$ and let $A_0 := A_{{\rm iso},2}^{\kappa_0,\mu_0}$ be the corresponding isotropic elasticity tensor. Fix $\mu >0$ and $0< \alpha_- < \beta^+ < +\infty$. Then the mapping $[\alpha_-,\beta^+] \ni \kappa \mapsto \mathcal{E}^{A_0}_{\rm Id}(A_{{\rm iso},2}^{\kappa,\mu})$ defined by~\eqref{eq:Energy} (for the specific case $\sigma = {\rm Id}$ and $\A = A_0$) is {\em strictly concave}.
\end{lemma}

\begin{proof}
  Using again~\eqref{energy-isotropic_3_kappa_mu}, we see that
  $$
  \frac{1}{3} \, \mathcal{E}^{A_0}_{\rm Id}(A_{{\rm iso},2}^{\kappa,\mu})
  =
  \kappa_0 -\frac{\big( \kappa_0-\kappa \big)^2}{\kappa_0+4\mu},
  $$
  which is obviously a strictly concave function of $\kappa$. In addition, the second derivative of this function with respect to $\kappa$ is (negative and) bounded away from 0 on the compact set $[\alpha_-,\beta^+]$.
\end{proof}

\begin{lemma}
\label{lem:strict-concave-mu}
Let $d=3$. Let $\kappa_0>0$, $\mu_0>0$ and let $A_0 := A_{{\rm iso},2}^{\kappa_0,\mu_0}$ be the corresponding isotropic elasticity tensor. Fix $\kappa>0$, $0< \alpha_- < \beta^+ < +\infty$ and consider some $1 \leq k,l \leq 3$ with $k \neq l$. Then the mapping $[\alpha_-/2,\beta^+/2] \ni \mu \to \mathcal{E}^{A_0}_{\sigma^{kl}}(A_{{\rm iso},2}^{\kappa,\mu})$ defined by~\eqref{eq:Energy} (for the specific case $\sigma = \sigma^{kl}$ and $\A = A_0$) is {\em strictly concave}.
\end{lemma}

\begin{proof}
  Using~\eqref{energy-isotropic_1}, we see that $\dps \mathcal{E}^{A_0}_{\sigma^{kl}}(A_{{\rm iso},2}^{\kappa,\mu}) = \mu_0 - f(\mu)$ with
  $$
  f(\mu) = \frac{2(\mu-\mu_0)^2}{2\mu_0 + \frac{8\mu+3\kappa}{6\mu+\kappa} \, \mu} = \frac{2(\mu-\mu_0)^2 \, (6\mu+\kappa)}{2\mu_0 \, (6\mu + \kappa) + (8\mu+3\kappa) \, \mu},
  $$
  and we are thus left with showing the strict convexity of $f$. A lengthy but straighforward computation shows that
  \begin{equation} \label{eq:def_f_prime}
  f'(\mu) = \frac{4(\mu-\mu_0) \, (6\mu+\kappa) + 12(\mu-\mu_0)^2}{2\mu_0 \, (6\mu + \kappa) + (8\mu+3\kappa) \, \mu} - \frac{2(\mu-\mu_0)^2 \, (6\mu+\kappa)}{\big[ 2\mu_0 \, (6\mu + \kappa) + (8\mu+3\kappa) \, \mu \big]^2} \, (12\mu_0 + 16\mu + 3\kappa),
  \end{equation}
  and that
  \begin{equation} \label{eq:def_f_der2}
  f''(\mu) = \frac{F(\mu)}{\big[ 2\mu_0 \, (6\mu + \kappa) + (8\mu+3\kappa) \, \mu \big]^3},
  \end{equation}
  with
  \begin{align*}
    F(\mu)
    &=
    \big[ 4(6\mu+\kappa) + 48(\mu-\mu_0) \big] \big[ 2\mu_0 \, (6\mu + \kappa) + (8\mu+3\kappa) \, \mu \big]^2
    \\
    & - \big[ 4(\mu-\mu_0) \, (6\mu+\kappa) + 12(\mu-\mu_0)^2 \big] \big[ 2\mu_0 \, (6\mu + \kappa) + (8\mu+3\kappa) \, \mu \big] (12\mu_0 + 16\mu + 3\kappa)
    \\
    & - \big[ 4(\mu-\mu_0) \, (6\mu+\kappa) + 12(\mu-\mu_0)^2 \big] \big[ 2\mu_0 \, (6\mu + \kappa) + (8\mu+3\kappa) \, \mu \big] (12\mu_0 + 16\mu + 3\kappa)
    \\
    & - 32(\mu-\mu_0)^2 \, (6\mu+\kappa) \big[ 2\mu_0 \, (6\mu + \kappa) + (8\mu+3\kappa) \, \mu \big]
    \\
    & + 4(\mu-\mu_0)^2 \, (6\mu+\kappa) \, (12\mu_0 + 16\mu + 3\kappa)^2,
  \end{align*}
  where the first two lines correspond to the derivative of the first term of~\eqref{eq:def_f_prime} and the last three lines correspond to the derivative of the second term of~\eqref{eq:def_f_prime}. Observing that the second and third lines of the above expression are identical, and considering $F(\mu)$ as a polynomial function in $\mu-\mu_0$, we get
  \begin{equation} \label{eq:F_trinome}
  F(\mu) = c_2 \, (\mu-\mu_0)^2 + c_1 \, (\mu-\mu_0) + c_0,
  \end{equation}
  where the coefficients $c_0$, $c_1$ and $c_2$ (which still depend on $\mu$) are given by 
  \begin{multline*}
    c_2
    =
    -24 \big[ 2\mu_0 \, (6\mu + \kappa) + (8\mu+3\kappa) \, \mu \big] (12\mu_0 + 16\mu + 3\kappa)
    \\ - 32 (6\mu+\kappa) \big[ 2\mu_0 \, (6\mu + \kappa) + (8\mu+3\kappa) \, \mu \big]
    + 4 (6\mu+\kappa) \, (12\mu_0 + 16\mu + 3\kappa)^2,
  \end{multline*}
  and
  \begin{align*}
    c_1
    &=
    48 \big[ 2\mu_0 \, (6\mu + \kappa) + (8\mu+3\kappa) \, \mu \big]^2
    - 8 (6\mu+\kappa) \big[ 2\mu_0 \, (6\mu + \kappa) + (8\mu+3\kappa) \, \mu \big] (12\mu_0 + 16\mu + 3\kappa),
    \\
    c_0
    &=
    4(6\mu+\kappa) \big[ 2\mu_0 \, (6\mu + \kappa) + (8\mu+3\kappa) \, \mu \big]^2.
  \end{align*}
  Introducing the notation
  $$
  \gamma_0
  =
  2\mu_0 \, (6\mu + \kappa) + (8\mu+3\kappa) \, \mu,
  \qquad
  \gamma_1
  =
  12\mu_0 + 16\mu + 3\kappa,
  \qquad
  \gamma_2
  =
  6\mu + \kappa,
  $$
  we get
  $$
  c_2
  =
  4 \big[ -6 \, \gamma_0 \, \gamma_1 - 8 \, \gamma_0 \, \gamma_2 + \gamma_1^2 \, \gamma_2 \big],
  \qquad
  c_1
  =
  48 \, \gamma_0^2 - 8 \, \gamma_0 \, \gamma_1 \, \gamma_2,
  \qquad
  c_0
  =
  4 \, \gamma_0^2 \, \gamma_2.
  $$
  To show that $f$ is strictly convex, we are left with showing that $F$ remains positive. We compute the discriminant $\Delta = c_1^2 - 4 \, c_2 \, c_0$ of the polynomial~\eqref{eq:F_trinome} of degree two:
  \begin{align}
    - \Delta
    &=
    4 \, c_2 \, c_0 - c_1^2
    \nonumber
    \\
    &=
    64 \, \gamma_0^2 \, \gamma_2 \big[ -6 \, \gamma_0 \, \gamma_1 - 8 \, \gamma_0 \, \gamma_2 + \gamma_1^2 \, \gamma_2 \big] - 64 \, \gamma_0^2 \big[ 6 \, \gamma_0 - \gamma_1 \, \gamma_2 \big]^2
    \nonumber
    \\
    &=
    64 \, \gamma_0^2 \big[ -6 \, \gamma_0 \, \gamma_1 \, \gamma_2 - 8 \, \gamma_0 \, \gamma_2^2 + \gamma_1^2 \, \gamma_2^2 - (6 \, \gamma_0 - \gamma_1 \, \gamma_2)^2 \big]
    \nonumber
    \\
    &=
    64 \, \gamma_0^2 \big[ - 36 \, \gamma_0^2 + 6 \, \gamma_0 \, \gamma_1 \, \gamma_2 - 8 \, \gamma_0 \, \gamma_2^2 \big]
    \nonumber
    \\
    &=
    128 \, \gamma_0^3 \big[ - 18 \, \gamma_0 + 3 \, \gamma_1 \, \gamma_2 - 4 \, \gamma_2^2 \big].
    \label{eq:discri}
  \end{align}
  Using the expression of $\gamma_0$, $\gamma_1$ and $\gamma_2$, we observe that
  \begin{align*}
    - 18 \, \gamma_0 + 3 \, \gamma_1 \, \gamma_2 - 4 \, \gamma_2^2
    &=
    3 \, (12\mu_0 + 16\mu + 3\kappa) \, (6\mu + \kappa) - 4 \, (6\mu + \kappa)^2 - 18 \, \big( 2\mu_0 \, (6\mu + \kappa) + (8\mu+3\kappa) \, \mu \big)
    \\
    &=
    3 \, (16\mu + 3\kappa) \, (6\mu + \kappa) - 4 \, (6\mu + \kappa)^2 - 18 \, (8\mu+3\kappa) \, \mu
    \\
    &=
    5 \, \kappa^2.
  \end{align*}
  Inserting this result in~\eqref{eq:discri}, we obtain $- \Delta = 640 \, \gamma_0^3 \, \kappa^2 > 0$. In addition, $F(\mu=\mu_0) = c_0 > 0$.

  We thus obtain that the polynomial function~\eqref{eq:F_trinome}, which reads $\dps F(\mu) = c_2 \left[ \left( \mu-\mu_0 + \frac{c_1}{2 \, c_2} \right)^2 - \frac{\Delta}{4 c_2^2} \right]$, remains of constant sign (since $\Delta < 0$) and actually remains positive (since $F(\mu_0) > 0$). In view of~\eqref{eq:def_f_der2}, we deduce that the function $f$ is strictly convex with respect to $\mu$.

  In addition, on the compact set $[\alpha_-/2,\beta^+/2]$, the continuous function $f''$ is (positive and) bounded away from 0. This concludes the proof of Lemma~\ref{lem:strict-concave-mu}.
\end{proof}

In what follows, we make the following assumption: 

\begin{ass} \label{ass:fondamental}
  Let $d=3$ and consider some $\A \in L^\infty(B,\mathcal{I})$. We assume that $\A$ is such that there exists some $\alpha_-$ and $\beta^+$ with $0< \alpha_- < \alpha$ and $\beta < \beta^+ < +\infty$ such that
  \begin{itemize}
  \item for any given $\mu \in [\alpha_-/2,\beta^+/2]$, the mapping $[\alpha_-,\beta^+] \ni \kappa \to \mathcal{E}^{\A}_{\rm Id}(A_{{\rm iso},2}^{\kappa,\mu})$ is strictly concave.
  \item for any given $\kappa \in [\alpha_-,\beta^+]$, the mapping $\dps [\alpha_-/2,\beta^+/2] \ni \mu \to \sum_{1 \leq k < l \leq 3} \mathcal{E}^{\A}_{\sigma^{kl}}(A_{{\rm iso},2}^{\kappa,\mu})$ is strictly concave.
  \end{itemize}
\end{ass}

In view of Lemmas~\ref{lem:strict-concave-lambda} and~\ref{lem:strict-concave-mu}, this assumption holds (for any $\alpha_-$ and $\beta^+$ such that $0< \alpha_- < \alpha$ and $\beta < \beta^+ < +\infty$) when $\A$ is a uniform isotropic field (i.e. $\A(x) = A_{{\rm iso},2}^{\kappa_0,\mu_0} \in \mathcal{I}$ for any $x \in B$). By a perturbative argument, it also holds when $\A$ is close (in the $L^\infty(B)$ norm) of such a field. An example falling within that situation is when
\begin{equation} \label{eq:perturb}
\forall x \in B, \quad \A(x) = A_{{\rm iso},2}^{\kappa_0,\mu_0} + \eta \, \C(x),
\end{equation}
for some $\C \in L^\infty(B,\mathcal{I})$, with $\eta$ sufficiently small. We indeed have the following result.


\begin{lemma} \label{lem:perturb}
Let $d=3$ and suppose that $\A$ is of the form~\eqref{eq:perturb} for some $\kappa_0 \in [\alpha,\beta]$, some $\mu_0 \in [\alpha/2,\beta/2]$ and some $\C \in L^\infty(B,\mathcal{I})$. Then,  for any $\alpha_-$ and $\beta^+$ such that $0< \alpha_- < \alpha$ and $\beta < \beta^+ < +\infty$, the field $\A$ satisfies Assumption~\ref{ass:fondamental} when $\eta$ is sufficiently small.
\end{lemma}

\begin{proof}
Let $0< \alpha_- < \alpha$ and $\beta < \beta^+ < +\infty$. Fix $\sigma \in \mathcal{S}$ and consider the embedded corrector problem~\eqref{p-corrector} with the interior elasticity field $\A$ and with the exterior elasticity field $A_{{\rm iso},2}^{\kappa,\mu}$, that we denote here $A^{\kappa,\mu}$ to simplify the notation. We denote $w_\sigma^{\A,\kappa,\mu}$ the solution to~\eqref{p-corrector}:
\begin{equation}\label{p-corrector_iso}
- {\rm div} \left[ \chi_B \, \A \left(\sigma + e\left(w_\sigma^{\A,\kappa,\mu}\right) \right) + \big(1-\chi_B\big) \, A^{\kappa,\mu} \left(\sigma + e\left(w_\sigma^{\A,\kappa,\mu}\right) \right) \right] = 0 \quad \text{in $\mathcal{D}'(\R^d)$},
\end{equation}
where we recall that $\chi_B$ is the characteristic function of the unit ball $B$. 

For any fixed $(\kappa,\mu) \in [\alpha_-,\beta^+] \times [\alpha_-/2,\beta^+/2]$, we have that $A^{\kappa,\mu} \in {\cal M}$, and the solution $w_\sigma^{\A,\kappa,\mu} \in V_0$ is continuous with respect to $\A \in {\cal M}$ in the $L^\infty(B)$ norm.

We now differentiate~\eqref{p-corrector_iso} with respect to $\kappa$ (the same arguments carry over with respect to $\mu$): denoting $w_{\sigma,1}^{\A,\kappa,\mu} = \partial_\kappa w_\sigma^{\A,\kappa,\mu}$, we have
\begin{equation}\label{p-corrector_iso_1}
- {\rm div} \left[ \chi_B \, \A \, e\left(w_{\sigma,1}^{\A,\kappa,\mu}\right) + \big(1-\chi_B\big) \, A^{\kappa,\mu} \, e\left(w_{\sigma,1}^{\A,\kappa,\mu}\right) \right] = {\rm div} \left[ \big(1-\chi_B\big) \, (\partial_\kappa A^{\kappa,\mu}) \left(\sigma + e\left(w_\sigma^{\A,\kappa,\mu}\right) \right) \right] \ \ \text{in $\mathcal{D}'(\R^d)$},
\end{equation}
with $\dps (\partial_\kappa A^{\kappa,\mu}) \tau = \frac{1}{3} \, ({\rm Tr} (\tau)) \, {\rm Id}$ for any $\tau \in \mathcal{S}$.

We again observe that, for any fixed $(\kappa,\mu) \in [\alpha_-,\beta^+] \times [\alpha_-/2,\beta^+/2]$, the solution $w_{\sigma,1}^{\A,\kappa,\mu} \in V_0$ is continuous with respect to $\A \in {\cal M}$ in the $L^\infty(B)$ norm.

We differentiate~\eqref{p-corrector_iso_1} with respect to $\kappa$: denoting $w_{\sigma,2}^{\A,\kappa,\mu} = \partial_\kappa w_{\sigma,1}^{\A,\kappa,\mu}$, we have
$$
- {\rm div} \left[ \chi_B \, \A \, e\left(w_{\sigma,2}^{\A,\kappa,\mu}\right) + \big(1-\chi_B\big) \, A^{\kappa,\mu} \, e\left(w_{\sigma,2}^{\A,\kappa,\mu}\right) \right] = 2 \, {\rm div} \left[ \big(1-\chi_B\big) \, (\partial_\kappa A^{\kappa,\mu}) \, e\left(w_{\sigma,1}^{\A,\kappa,\mu}\right) \right]\quad \text{in $\mathcal{D}'(\R^d)$},
$$
and we again have that, for any fixed $(\kappa,\mu) \in [\alpha_-,\beta^+] \times [\alpha_-/2,\beta^+/2]$, the solution $w_{\sigma,2}^{\A,\kappa,\mu} \in V_0$ is continuous with respect to $\A \in {\cal M}$ in the $L^\infty(B)$ norm.

We now compute ${\cal F}(\A,\kappa) := \mathcal{E}_\sigma^{\A}(A^{\kappa,\mu})$ using~\eqref{elseE_2}, for a fixed $\mu \in [\alpha_-/2,\beta^+/2]$. Differentiating the expression with respect to $\kappa$, we obtain
$$
|B| \, \frac{\partial {\cal F}}{\partial \kappa}(\A,\kappa)
=
\int_B \sigma \cdot \A \, e\left(w_{\sigma,1}^{\A,\kappa,\mu}\right) - \int_{\mathbb{S}} w_{\sigma,1}^{\A,\kappa,\mu} \cdot ((A^{\kappa,\mu}\sigma) \, n) - \int_{\mathbb{S}} w_\sigma^{\A,\kappa,\mu} \cdot ((\partial_\kappa A^{\kappa,\mu}\sigma) \, n)
$$
and$$
|B| \, \frac{\partial^2 {\cal F}}{\partial \kappa^2}(\A,\kappa)
=
\int_B \sigma \cdot \A \, e\left(w_{\sigma,2}^{\A,\kappa,\mu}\right) - \int_{\mathbb{S}} w_{\sigma,2}^{\A,\kappa,\mu} \cdot ((A^{\kappa,\mu}\sigma) \, n) - 2 \int_{\mathbb{S}} w_{\sigma,1}^{\A,\kappa,\mu} \cdot ((\partial_\kappa A^{\kappa,\mu}\sigma) \, n).
$$
The continuities established above imply that, for any fixed $\kappa \in [\alpha_-,\beta^+]$, the quantity $\dps \frac{\partial^2 {\cal F}}{\partial \kappa^2}(\A,\kappa)$ is a function continuous with respect to $\A \in {\cal M}$ in the $L^\infty(B)$ norm.

We have shown in Lemma~\ref{lem:strict-concave-lambda} that, for $\sigma = {\rm Id}$ and $\A = A_{{\rm iso},2}^{\kappa_0,\mu_0}$, the quantity $\dps \frac{\partial^2 {\cal F}}{\partial \kappa^2}(\A,\kappa)$ is negative and bounded away from 0. As a consequence, for $\A$ of the form~\eqref{eq:perturb} for a sufficiently small parameter $\eta$, we again have that $\dps \frac{\partial^2 {\cal F}}{\partial \kappa^2}(\A,\kappa)$ is negative, which means that the mapping $[\alpha_-,\beta^+] \ni \kappa \to \mathcal{E}^{\A}_{\rm Id}(A_{{\rm iso},2}^{\kappa,\mu})$ is strictly concave. This concludes the proof of Lemma~\ref{lem:perturb}.
\end{proof}

\subsection{Self-consistent tensor in isotropic elasticity}

In general, the existence of a fourth-order tensor satisfying the self-consistent equation~\eqref{tensor4} is an open question. However, for isotropic materials, we have the following positive result.

\begin{proposition} \label{self-consistent}
Let $d=3$ and consider a sequence $(\A^N)_{N \in \N} \subset L^\infty(B,\mathcal{I})$. Since $\A^N(x)$ is isotropic, it satisfies
$$
\forall x \in B, \quad \forall \sigma \in \cS, \quad \A^N(x) \sigma = 2\mu^N(x) \, \sigma + \lambda^N(x) \, ({\rm Tr} \, \sigma) \, {\rm Id}
$$
where $\mu^N$ and $\lambda^N$ are two functions in $L^\infty(B)$. Assume in addition that the sequence $(\A^N)_{N\in\N}$ G-converges to a constant isotropic elasticity tensor $A^\star\in \mathcal{I}$, the Lam\'e coefficients of which are denoted by $\lambda^\star$ and $\mu^\star$. We also assume that, for any $N \in \N$, the field $\A^N$ satisfies Assumption~\ref{ass:fondamental}. 

Then, for all $N \in \N$, there exists $A_4^N \in \mathcal{I}$, with Lam\'e coefficients $\lambda_4^N$ and $\mu_4^N$, such that
\begin{equation}\label{eq:mu}
\mu_4^N = \frac{1}{3} \sum_{1 \leq i <  j \leq 3} \mathcal{E}_{\sigma^{ij}}^{\A^N}(A_4^N) \quad \text{and} \quad 2\mu_4^N + 3\lambda_4^N = \frac{1}{3} \mathcal{E}_{\rm Id}^{\A^N}(A_4^N).
\end{equation}
In addition, we have
\begin{equation}\label{converge4}
\mu_4^N \mathop{\longrightarrow}_{N\to +\infty} \mu^\star \quad \mbox{and} \quad \lambda_4^N \mathop{\longrightarrow}_{N\to +\infty} \lambda^\star.
\end{equation}
\end{proposition}

Before proceeding to the proof, let us motivate the self-consistent equations~\eqref{eq:mu}. Consider the simple case when the tensor $\A$ is constant in $B$ and equal to some $A \in \cM$. In that case, the unique solution to the embedded corrector problem~\eqref{p-corrector} is of course $w_\sigma^{\A,A} = 0$ for any $\sigma \in \cS$ and $\mathcal{E}_{\sigma}^{\A}(A) = \sigma \cdot A \sigma$ in view of~\eqref{elseE_1}. Assuming now that $A$ is isotropic (with Lam\'e coefficients $\lambda$ and $\mu$), we get
\begin{equation} \label{eq:cas_simple}
  \mathcal{E}_{\sigma}^{\A}(A) = 2\mu \, \sigma \cdot \sigma + \lambda \, ({\rm Tr} \, \sigma)^2.
\end{equation}
In order to determine $\mu$ from $\mathcal{E}_{\sigma}^{\A}(A)$, let us consider trace-free matrices $\sigma$. For any $1 \leq i <  j \leq 3$, we get $\mathcal{E}_{\sigma^{ij}}^{\A}(A) = 2\mu \, \sigma^{ij} \cdot \sigma^{ij} = \mu$. This would thus motivate, in the heterogeneous case, a self-consistent equation of the form $\mu_4^N = \mathcal{E}_{\sigma^{ij}}^{\A^N}(A_4^N)$ for any fixed $1 \leq i < j \leq 3$. To obtain an equation independent of the specific choice of the couple $(i,j)$, it is natural to sum the previous relation on the three possible couples $(i,j)$, hence the first equation of~\eqref{eq:mu}.

In order to determine $\lambda$, we return to the homogeneous case~\eqref{eq:cas_simple} and set $\sigma = {\rm Id}$. We then have $\mathcal{E}_{\rm Id}^{\A}(A) = 6\mu+9\lambda$. This motivates, in the heterogeneous case, the second equation of~\eqref{eq:mu}.

\begin{proof}
Instead of working with the Lam\'e coefficients $\mu$ and $\lambda$, we are going to again work with $\mu$ and $\kappa = 2\mu+3\lambda$, as pointed out at the beginning of Section~\ref{sec:strict}. Proving the existence of a tensor $A_4^N \in \mathcal{I}$ with Lam\'e coefficients $\lambda_4^N$ and $\mu_4^N$ satisfying~\eqref{eq:mu} is equivalent to proving the existence of $\mu_4^N \in [\alpha/2, \beta/2]$ and $\kappa_4^N \in [\alpha, \beta]$ such that 
\begin{equation}\label{eq:mukappa}
\mu_4^N = \frac{1}{3} \sum_{1\leq i <  j \leq 3} \mathcal{E}_{\sigma^{ij}}^{\A_N}\left(A_{{\rm iso},2}^{\kappa_4^N,\mu_4^N}\right) \quad \text{and} \quad \kappa_4^N = \frac{1}{3} \mathcal{E}_{\rm Id}^{\A_N}\left(A_{{\rm iso},2}^{\kappa_4^N,\mu_4^N}\right).
\end{equation}
Let us denote by $\overline{A} := A_{{\rm iso},2}^{\beta,\beta/2}$ and by $\underline{A} := A_{{\rm iso},2}^{\alpha,\alpha/2}$. For any $\sigma \in \cS$, we find that
\begin{align*}
  \sigma \cdot \A^N(x) \sigma
  &=
  2\mu^N(x) \, \left( \sigma \cdot \sigma - \frac{1}{3} ({\rm Tr} \, \sigma)^2 \right) + \frac{\kappa^N(x)}{3} \, ({\rm Tr} \, \sigma)^2
  \\
  &=
  2\mu^N(x) \, \left| \sigma - \frac{{\rm Tr} \, \sigma}{3} \, {\rm Id} \right|^2 + \frac{\kappa^N(x)}{3} \, ({\rm Tr} \, \sigma)^2.
\end{align*}
Since $\A^N(x) \in \mathcal{I}$ and in view of~\eqref{eq:borne_I_bis}, we have $ \sigma \cdot \underline{A} \sigma \leq \sigma \cdot \A^N(x) \sigma \leq \sigma \cdot \overline{A} \sigma$ for any $\sigma \in \cS$. 

Let $\kappa \in [\alpha_-, \beta^+]$ and $\mu \in [\alpha_-/2, \beta^+/2]$, where $\alpha_-$ and $\beta^+$ are constants such that Assumption~\ref{ass:fondamental} is satisfied. In view of~\eqref{eq:E}, we deduce from the above estimates that, for any $v \in V_0$ and $\sigma \in \cS$,
$$
 \mathcal{E}_\sigma^{\underline{A},A_{{\rm iso},2}^{\kappa,\mu}}(v) \leq \mathcal{E}_\sigma^{\A^N,A_{{\rm iso},2}^{\kappa,\mu}}(v) \leq \mathcal{E}_\sigma^{\overline{A},A_{{\rm iso},2}^{\kappa,\mu}}(v). 
$$
Minimizing over $v \in V_0$, we hence obtain that
\begin{equation} \label{eq:covid8}
\mathcal{E}_\sigma^{\underline{A}}(A_{{\rm iso},2}^{\kappa,\mu}) \leq \mathcal{E}_\sigma^{\A^N}(A_{{\rm iso},2}^{\kappa,\mu}) \leq \mathcal{E}_\sigma^{\overline{A}}(A_{{\rm iso},2}^{\kappa,\mu}). 
\end{equation}
Let us define
\begin{equation}\label{F1}
F_{\A^N}(\mu,\kappa) := \frac{1}{3} \sum_{1 \leq i <  j \leq 3} \mathcal{E}_{\sigma^{ij}}^{\A^N}(A_{{\rm iso},2}^{\kappa,\mu}) - \mu,
\end{equation}
so that the first equation in~\eqref{eq:mukappa} reads $F_{\A^N}(\mu_4^N,\kappa_4^N) = 0$. We also define
\begin{equation}\label{FdownFup}
F_{\underline{A}}(\mu,\kappa) := \frac{1}{3} \sum_{1 \leq i< j \leq 3} \mathcal{E}_{\sigma^{ij}}^{\underline{A}}(A_{{\rm iso},2}^{\kappa,\mu}) - \mu, \qquad F_{\overline{A}}(\mu,\kappa) := \frac{1}{3} \sum_{1 \leq i < j \leq 3} \mathcal{E}_{\sigma^{ij}}^{\overline{A}}(A_{{\rm iso},2}^{\kappa,\mu}) - \mu.
\end{equation}
We obviously infer from~\eqref{eq:covid8} that
\begin{equation}\label{F-F}
\forall \mu \in [\alpha_-/2,\beta^+/2], \ \ \forall \kappa \in [\alpha_-,\beta^+], \quad  F_{\underline{A}}(\mu,\kappa) \leq F_{\A^N}(\mu,\kappa) \leq F_{\overline{A}}(\mu,\kappa). 
\end{equation}
Similarly, let us define
\begin{equation}\label{eq:G}
G_{\A^N}(\mu,\kappa) := \frac{1}{3} \mathcal{E}_{\rm Id}^{\A^N}\left(A_{{\rm iso},2}^{\kappa,\mu}\right) - \kappa,
\end{equation}
so that the second equation in~\eqref{eq:mukappa} reads $G_{\A^N}(\mu_4^N,\kappa_4^N) = 0$. We similarly introduce $G_{\underline{A}}$ and $G_{\overline{A}}$, and again infer from~\eqref{eq:covid8} that
\begin{equation}\label{G-G}
\forall \mu \in [\alpha_-/2,\beta^+/2], \ \ \forall \kappa \in [\alpha_-,\beta^+], \quad G_{\underline{A}}(\mu,\kappa) \leq G_{\A^N}(\mu,\kappa) \leq G_{\overline{A}}(\mu,\kappa).
\end{equation}
The proof proceeds in two steps.

\medskip


\noindent
{\bf Step 1: existence of a fixed point.} 
To compute $F_{\underline{A}}$, we consider the Eshelby's problem~\eqref{Eshelby-isotropic} with $\sigma = \sigma^{ij}$ for some $1\leq i <j \leq 3$, with the interior isotropic tensor $\underline{A}$ (which corresponds to the Lam\'e coefficients $\mu_0 = \alpha/2$ and $\lambda_0 = 0$) and the exterior isotropic tensor $A_{{\rm iso},2}^{\kappa,\mu}$. Using~\eqref{energy-isotropic_1}, we find
\begin{equation}\label{F_value}
F_{\underline{A}}(\mu,\kappa) = (\alpha/2-\mu) \left( 1 - \frac{2(\alpha/2-\mu)}{\alpha + \frac{8\mu+3\kappa}{6\mu+\kappa} \, \mu} \right),
\end{equation}
and we likewise have
\begin{equation}\label{F_value_bis}
F_{\overline{A}}(\mu,\kappa) = (\beta/2-\mu) \left( 1 - \frac{2(\beta/2-\mu)}{\beta + \frac{8\mu+3\kappa}{6\mu+\kappa} \, \mu} \right).
\end{equation}
Fix $\kappa$ in $[\alpha_-,\beta^+]$. For any $\mu \in [\alpha_-/2,\beta^+/2]$, we have $\dps \alpha + \frac{8\mu+3\kappa}{6\mu+\kappa} \, \mu > \alpha > 2(\alpha/2-\mu)$, hence, in view of~\eqref{F_value}, we have $F_{\underline{A}}(\mu,\kappa) < 0$ for any $\mu \in (\alpha/2,\beta^+/2]$, $F_{\underline{A}}(\mu,\kappa) > 0$ for any $\mu \in [\alpha_-/2,\alpha/2)$ and the equation $F_{\underline{A}}(\mu,\kappa) = 0$ has a unique solution in $[\alpha_-/2,\beta^+/2]$ which is $\mu = \alpha/2$. Likewise, for any $\mu \in [\alpha_-/2,\beta^+/2]$, we have $\dps \beta + \frac{8\mu+3\kappa}{6\mu+\kappa} \, \mu > \beta > 2(\beta/2-\mu)$, hence, in view of~\eqref{F_value_bis}, we have $F_{\overline{A}}(\mu,\kappa) > 0$ for any $\mu \in [\alpha_-/2,\beta/2)$, $F_{\overline{A}}(\mu,\kappa) < 0$ for any $\mu \in (\beta/2,\beta^+/2]$ and the equation $F_{\overline{A}}(\mu,\kappa) = 0$ has a unique solution in $[\alpha_-/2,\beta^+/2]$ which is $\mu = \beta/2$. The bound~\eqref{F-F} thus implies that there exists some $\widehat{\mu}(\kappa) \in [\alpha/2,\beta/2]$ such that $F_{\A^N}(\widehat{\mu}(\kappa),\kappa) = 0$. In addition, in view of $\A^N$ satisfying Assumption~\ref{ass:fondamental} and of the definition~\eqref{F1}, the map $[\alpha_-/2,\beta^+/2] \ni \mu \mapsto F_{\A^N}(\mu,\kappa)$ is strictly concave. For a fixed $\kappa$, the equation $F_{\A^N}(\mu,\kappa) = 0$ thus has a unique solution in $[\alpha/2,\beta/2]$ that we denote $\widehat{\mu}(\kappa) \in [\alpha/2,\beta/2]$ (uniqueness can e.g. be shown by considering the various possibilities for the sign on $[\alpha,\beta]$ of the derivative of $\mu \mapsto F_{\A^N}(\mu,\kappa)$). In addition, the mapping $[\alpha,\beta] \ni \kappa \mapsto \widehat{\mu}(\kappa)$ is continuous.

\medskip

In a similar manner, we compute $G_{\underline{A}}$ by considering the Eshelby's problem~\eqref{Eshelby-isotropic} with $\sigma = {\rm Id}$ with the interior isotropic tensor $\underline{A}$ and the exterior isotropic tensor $A_{{\rm iso},2}^{\kappa,\mu}$. Using~\eqref{energy-isotropic_3}, we find
\begin{equation}\label{G_value}
G_{\underline{A}}(\mu,\kappa) = (\alpha-\kappa) \left( 1 - \frac{\alpha-\kappa}{\alpha+4\mu} \right) 
\end{equation}
and we likewise have
\begin{equation}\label{G_value_bis}
G_{\overline{A}}(\mu,\kappa) = (\beta-\kappa) \left( 1 - \frac{\beta-\kappa}{\beta+4\mu} \right).
\end{equation}
Fix $\mu$ in $[\alpha_-/2,\beta^+/2]$. For any $\kappa \in [\alpha_-,\beta^+]$, we have $\alpha+4\mu > \alpha-\kappa$, hence, in view of~\eqref{G_value}, we have $G_{\underline{A}}(\mu,\kappa) < 0$ for any $\kappa \in (\alpha,\beta^+]$, $G_{\underline{A}}(\mu,\kappa) > 0$ for any $\kappa \in [\alpha_-,\alpha)$ and the equation $G_{\underline{A}}(\mu,\kappa) = 0$ has a unique solution in $[\alpha_-,\beta^+]$ which is $\kappa = \alpha$. Likewise, any $\kappa \in [\alpha_-,\beta^+]$, we have $\beta+4\mu > \beta-\kappa$, hence, in view of~\eqref{G_value_bis}, we have $G_{\overline{A}}(\mu,\kappa) > 0$ for any $\kappa \in [\alpha_-,\beta)$, $G_{\overline{A}}(\mu,\kappa) < 0$ for any $\kappa \in (\beta,\beta^+]$ and the equation $G_{\overline{A}}(\mu,\kappa) = 0$ has a unique solution in $[\alpha_-,\beta^+]$ which is $\kappa = \beta$. The bound~\eqref{G-G} thus implies that there exists some $\widehat{\kappa}(\mu) \in [\alpha,\beta]$ such that $G_{\A^N}(\mu,\widehat{\kappa}(\mu)) = 0$. In addition, in view of $\A^N$ satisfying Assumption~\ref{ass:fondamental} and of the definition~\eqref{eq:G}, the map $[\alpha_-,\beta^+] \ni \kappa \mapsto G_{\A^N}(\mu,\kappa)$ is strictly concave. For a fixed $\mu$, the equation $G_{\A^N}(\mu,\kappa)= 0$ thus has a unique solution in $[\alpha,\beta]$ that we denote $\widehat{\kappa}(\mu) \in [\alpha,\beta]$. In addition, the mapping $[\alpha/2,\beta/2] \ni \mu \mapsto \widehat{\kappa}(\mu)$ is continuous.

\medskip

Consider now the mapping $H: [\alpha/2,\beta/2] \times [\alpha,\beta] \ni (\mu,\kappa) \mapsto (\widehat{\mu}(\kappa),\widehat{\kappa}(\mu))$. Since $H$ is a continuous mapping with values in $[\alpha/2,\beta/2] \times [\alpha,\beta]$, we are in position to use Brouwer's fixed point theorem, which shows that there exists $(\mu_4^N,\kappa_4^N) \in [\alpha/2,\beta/2] \times [\alpha,\beta]$ such that $H(\mu_4^N,\kappa_4^N) = (\mu_4^N,\kappa_4^N)$. By construction, $(\mu_4^N, \kappa_4^N)$ satisfies the fixed point equations~\eqref{eq:mukappa}, which yields the claimed existence result. 

\medskip

\noindent
{\bf Step 2: limit when $N \to \infty$.}
We now prove that $\dps \mu_4^N \mathop{\longrightarrow}_{N\to +\infty} \mu^\star$ and $\dps \kappa_4^N \mathop{\longrightarrow}_{N\to +\infty} \kappa^\star$, where of course $\kappa^\star = 2\mu^\star+3\lambda^\star$. Since the sequence $(\mu_4^N, \kappa_4^N)$ is bounded, there exists $\mu_\infty \in [\alpha/2,\beta/2]$ and $\kappa_\infty \in [\alpha,\beta]$ such that (up to the extraction of a subsequence that we do not write explicitly)
$$
\mu_4^N \mathop{\longrightarrow}_{N\to +\infty} \mu_\infty \quad \mbox{and} \quad \kappa_4^N \mathop{\longrightarrow}_{N\to +\infty} \kappa_\infty. 
$$
Using Lemma~\ref{converge-in-out}, we have, for any $\sigma \in \cS$, that $\dps w_\sigma^{\A^N,A_{{\rm iso},2}^{\kappa_4^N,\mu_4^N}}$ converges to $\dps w_\sigma^{A^\star,A_{{\rm iso},2}^{\kappa_\infty,\mu_\infty}}$, weakly in $H^1_{\rm loc}(\R^3)^3$ and therefore strongly in $L^2(\mathbb{S})^3$. Lemma~\ref{converge-in-out} also yields that
$$
\A^N \left( \sigma + e\left(w_\sigma^{\A^N,A_{{\rm iso},2}^{\kappa_4^N,\mu_4^N}}\right) \right) \mathop{\rightharpoonup}_{N\to +\infty} A^\star \left( \sigma + e\left(w_\sigma^{A^\star,A_{{\rm iso},2}^{\kappa_\infty,\mu_\infty}}\right) \right) \quad \text{weakly in $L^2_{\rm loc}(\R^3)^{3 \times 3}$}.
$$
Using~\eqref{elseE_2}, we thus obtain that
$$
\forall \sigma \in \cS, \qquad \mathcal{E}_\sigma^{\A^N}\left(A_{{\rm iso},2}^{\kappa_4^N,\mu_4^N}\right) \mathop{\longrightarrow}_{N\to +\infty} \mathcal{E}_\sigma^{A^\star}\left(A_{{\rm iso},2}^{\kappa_\infty,\mu_\infty}\right).
$$
We can hence pass to the limit $N \to \infty$ in~\eqref{eq:mukappa} and obtain that
\begin{equation} \label{eq:covid10}
\mu_\infty = \frac{1}{3} \sum_{1\leq i< j \leq 3} \mathcal{E}_{\sigma^{ij}}^{A^\star}\left(A_{{\rm iso},2}^{\kappa_\infty,\mu_\infty} \right) \quad \mbox{and} \quad \kappa_\infty = \frac{1}{3} \, \mathcal{E}_{\rm Id}^{A^\star}\left(A_{{\rm iso},2}^{\kappa_\infty,\mu_\infty} \right).
\end{equation}
We first consider the equation on the left of~\eqref{eq:covid10}. Using the same comparison principle as the one used to obtain~\eqref{eq:covid8}, we note that its right-hand side is an increasing function of $\mu^\star$. Furthermore, using~\eqref{energy-isotropic_1}, this left equation reads
$$
\mu_\infty = \mu^\star - \frac{2(\mu_\infty-\mu^\star)^2}{2\mu^\star + \frac{8\mu_\infty+3\kappa_\infty}{6\mu_\infty+\kappa_\infty} \, \mu_\infty}.
$$
The right-hand side being an increasing function of $\mu^\star$, this equation implies that $\mu^\star = \mu_\infty$.

We next turn to the equation on the right of~\eqref{eq:covid10}. Using again the same comparison principle as the one used to obtain~\eqref{eq:covid8}, we note that its right-hand side is an increasing function of $\kappa^\star$. Furthermore, using~\eqref{energy-isotropic_3}, this right equation reads
$$
\kappa_\infty
=
\kappa^\star - \frac{(\kappa^\star-\kappa_\infty)^2}{\kappa^\star+4\mu_\infty}.
$$
The right-hand side being an increasing function of $\kappa^\star$, this equation implies that $\kappa^\star = \kappa_\infty$. This concludes the proof of Proposition~\ref{self-consistent}.
\end{proof}

\medskip

We conclude this section by observing that the definition~\eqref{tensor1} of $A_1^N$ can also be adapted to the context of isotropic elasticity, by introducing
\begin{equation}\label{tensor1_iso}
A_{1,{\rm iso}}^N \in \underset{A \in \mathcal{I}}{\mathop{\rm argmax}} \ \mathcal{E}^{\A^N}(A), 
\end{equation}
where we recall that $\mathcal{I} \ni A \mapsto \mathcal{E}^{\A^N}(A)$ is defined by~\eqref{E-sum}. In dimension $d=3$, if $\A^N$ satisfies Assumption~\ref{ass:fondamental}, then this mapping is strictly concave (it can be written as the sum of a strictly concave function and of a concave function, in view of Lemma~\ref{concave}). There hence exists a unique $A_{1,{\rm iso}}^N$ in $\mathcal{I}$ satisfying~\eqref{tensor1_iso}. Similarly to~\eqref{tensor2} and~\eqref{tensor3}, we can also introduce $A_{2,{\rm iso}}^N \in \widetilde{\cM}$ defined by
\begin{equation}\label{tensor2_iso}
\forall \sigma \in \cS, \qquad A_{2,{\rm iso}}^N \sigma = \frac{1}{|B|} \int_B \A^N \left(\sigma+ e\left(w_\sigma^{\A^N ,A_{1,{\rm iso}}^N}\right)\right)
\end{equation}
and $A^N_{3,{\rm iso}} \in \widetilde{\cM}$ such that 
\begin{equation}\label{tensor3_iso}
\forall \sigma \in \cS, \qquad \sigma \cdot A^N_{3,{\rm iso}} \sigma = \mathcal{E}_\sigma^{\A^N}(A^N_{1,{\rm iso}}).
\end{equation}
We then have the following proposition (which is not restricted to the case when $d=3$), the proof of which follows the same arguments as the proof of Proposition~\ref{prop:convergence}.

\begin{proposition} \label{prop:convergence_iso}
Let $(\A^N)_{N \in \N} \subset L^\infty(B,\mathcal{M})$ a family of tensors which converges in $B$ in the sense of homogenization to a constant tensor $A^\star\in \cM$. We furthermore assume that $A^\star$ is isotropic, and hence that $A^\star\in \mathcal{I}$. Let $A_{1,{\rm iso}}^N$, $A_{2,{\rm iso}}^N$ and $A_{3,{\rm iso}}^N$ be respectively defined by~\eqref{tensor1_iso}, \eqref{tensor2_iso} and~\eqref{tensor3_iso}. Then, we have 
$$
A_{1,{\rm iso}}^N \mathop{\longrightarrow}_{N\to +\infty} A^\star, \qquad A_{2,{\rm iso}}^N \mathop{\longrightarrow}_{N\to +\infty} A^\star \qquad \mbox{and} \qquad A_{3,{\rm iso}}^N \mathop{\longrightarrow}_{N\to +\infty} A^\star.  
$$
\end{proposition}

\section*{Acknowledgements}

The work of FL is partially supported by ONR under Grant N00014-20-1-2691 and by EOARD under grant FA8655-20-1-7043. FL acknowledges the continuous support from these two agencies. SX gratefully acknowledges the support from Labex MMCD (Multi-Scale Modelling \& Experimentation of Materials for Sustainable Construction; French government grant ANR-11-LABX-0022 managed by ANR within the frame of the national program Investments for the Future). The authors warmly thank Eric Canc\`es for several stimulating discussions in the early stages of this work.

\appendix

\section{Korn inequalities} \label{sec:appendix}

In this appendix, we collect some auxiliary lemmas, related to Korn's inequalities and rigid displacements, which are used throughout the article. We refer to~\cite{ciarlet} for a comprehensive presentation (see also the elegant review in~\cite{DV}). For any domain $D\subset \R^d$, we denote by
\begin{equation}\label{eq:RD}
 \cR(D) := \left\{ v\in H^1(D)^d, \quad e(v) = 0 \ \ \mbox{in $D$} \right\}
\end{equation}
the set of rigid displacements of $D$. 

\medskip

\begin{lemma}[Rigid displacements, see~\cite{ciarlet}] \label{lem:rigid}
Let $D \subset \R^d$ be an open, connex and bounded domain of $\R^d$. Let $u \in H^1(D)^d$. Then, $u\in \cR(D)$ if and only if there exists $b \in \R^d$ and $M \in \R^{d\times d}$ such that $M^T = -M$ and 
$$
\forall x\in D, \quad u(x) = Mx + b.
$$
\end{lemma}

\begin{lemma}[See~\cite{ciarlet}] \label{simple}
Let $D\subset \R^d$ be an open domain and let $u\in L^2_{\rm loc}(D)^d$ so that $\nabla u \in L^2(D)^{d\times d}$. Then, we have 
$$
\|e(u)\|_{L^2(D)^{d\times d}} \leq \|\nabla u \|_{L^2(D)^{d \times d}}, \qquad \| {\rm div} \ u \|_{L^2(D)} \leq \sqrt{d} \ \|\nabla u \|_{L^2(D)^{d \times d}}. 
$$
\end{lemma}

In this article, we only need the Korn's inequality in $H_0^1(D)$ (and not in $H^1(D)$), which is stated as follows:




\begin{lemma}[Korn's inequality in $H_0^1$, see~\cite{ciarlet}]
\label{Korn2}
Let $D \subset \R^d$ be a regular domain. For any $u \in H_0^1(D)$, we have
$$
\|\nabla u\|_{L^2(D)^{d\times d}} \leq \sqrt{2} \ \|e( u) \|_{L^2(D)^{d\times d}}.
$$
\end{lemma} 
\bibliographystyle{plain}
\bibliography{elasticity_FL}

\end{document}